\documentclass[]{amsart}

\usepackage{lmodern}
\usepackage[leqno]{amsmath}
\usepackage{amssymb,amsfonts,xfrac,MnSymbol}
\usepackage{amsthm}
\usepackage{cite}
\usepackage{hyperref}
\usepackage{todonotes}
\usepackage[shortlabels]{enumitem}
\usepackage{graphicx}
\usepackage{tikz-cd, tikz}
\usepackage{color}
\usepackage{import}
\usepackage{braket}
\usepackage{mathtools}
\usepackage{float}
\usepackage{stmaryrd}

\numberwithin{equation}{section}

\usepackage{marginnote}
\newcounter{mynote}% a new counter for use in margin notes

\newtheorem{theorem}{Theorem}[section]

\newtheorem{proposition}[theorem]{Proposition}
\newtheorem{lemma}[theorem]{Lemma}
\newtheorem{corollary}[theorem]{Corollary}

\newtheorem*{theorem*}{Theorem}
\newtheorem*{claim*}{Claim}
\newtheorem*{proposition*}{Proposition}
\newtheorem*{lemma*}{Lemma}
\newtheorem*{corollary*}{Corollary}

\newtheorem{theoremA}{Theorem}

\theoremstyle{definition}
\newtheorem{definition}[theorem]{Definition}

\newtheorem{remark}[theorem]{Remark}
\newtheorem{example}[theorem]{Example}

\newtheorem{conjecture}[theorem]{Conjecture}

\newtheorem*{definition*}{Definition}
\newtheorem*{observation*}{Observation}
\newtheorem*{remark*}{Remark}
\newtheorem*{example*}{Example}
\newtheorem*{question*}{Question}
\newtheorem*{exercise*}{Exercise}
\newtheorem*{fact*}{Fact}
\newtheorem*{notation*}{Notation}

\newcommand{\bbE}{\mathbb{E}}

\newcommand{\bbN}{\mathbb{N}}

\newcommand{\bbP}{\mathbb{P}}

\newcommand{\bbZ}{\mathbb{Z}}

\newcommand{\calG}{\mathcal{G}}

\newcommand{\calI}{\mathcal{I}}

\newcommand{\calM}{\mathcal{M}}
\newcommand{\calN}{\mathcal{N}}

\newcommand{\calP}{\mathcal{P}}

\newcommand{\calS}{\mathcal{S}}

\newcommand{\usigma}{\underline{\sigma}}

\newcommand{\ualpha}{\underline{\alpha}}

\newcommand{\inj}{\hookrightarrow}

\newcommand{\ii}{^{-1}}

\newcommand{\gen}[1]{\left< #1 \right>}

\newcommand{\fpf}{fixed-point-free~}
\newcommand{\FPF}{\mathcal{F}}
\newcommand{\FR}[1]{{#1}^{(\infty)}}

\newcommand{\comm}{\dot\simeq}

\DeclareMathOperator{\Stab}{Stab}

\DeclareMathOperator{\id}{id}
\DeclareMathOperator{\Aut}{Aut}

\DeclareMathOperator{\Sym}{Sym}
\DeclareMathOperator{\Alt}{Alt}

\DeclareMathOperator{\BMW}{BMW}
\DeclareMathOperator{\iBMW}{BMW_{inv}}

\title{Counting lattices in products of trees}
\author{Nir Lazarovich, Ivan Levcovitz and Alex Margolis}
\thanks{NL is supported by the Israel Science Foundation (grant no. 1562/19), and by the German-Israeli Foundation for Scientific Research and Development.}
\begin{document}
	\maketitle
\begin{abstract}
    A BMW group of degree $(m,n)$ is a group that acts simply transitively on vertices of the product of two regular trees of degrees $m$ and $n$.
    We show that the number of commensurability classes of BMW groups of degree $(m,n)$ is bounded between $(mn)^{\alpha mn}$ and $(mn)^{\beta mn}$ for some $0<\alpha<\beta$. In fact, 
    we show that 
    the same bounds hold for virtually simple BMW groups. 
    We introduce a random model for BMW groups of degree $(m,n)$ and show that asymptotically almost surely a random BMW group in this model is irreducible and hereditarily just-infinite.
\end{abstract}

\section{Introduction}
Given $n\in \bbN$, let $T_n$ denote the regular tree of valence $n$.
A \emph{BMW group of degree $(m,n)$} is a subgroup of $\Aut(T_m)\times \Aut(T_n)$ that  acts simply transitively on the vertex set of  $T_m\times T_n$.
Using these groups,
Wise \cite{wise2007Complete} and Burger-Mozes \cite{burgermozes2000lattices} produced the first examples of non-residually finite and virtually simple CAT(0) groups respectively.
BMW groups have been extensively studied and have rich connections to the study of automata groups  and commensurators (see Caprace's survey \cite{Caprace-survey}).

In this paper, our goal is two-fold: (1) estimate the number of BMW groups 
and virtually simple BMW groups
up to abstract commensurability, and (2) define and study a random model for BMW groups. 
Although conceptually related, the two parts are independently presented.

\subsection*{Counting BMW groups.}
Let $\BMW(m,n)$ be the set of all BMW groups of degree $(m,n)$ up to conjugacy in $\Aut(T_m)\times \Aut(T_n)$. 
Let $\comm$ be the equivalence relation of abstract commensurability, i.e., the groups $\Gamma$ and $\Lambda$ satisfy $\Gamma \comm \Lambda$ if they have isomorphic finite index subgroups.
In analogy to counting results for hyperbolic manifolds (see Remark \ref{rem: counting hyperbolic manifolds}), Caprace \cite[Problem 4.26]{Caprace-survey} asks for an estimate on the number of abstract commensurability classes of BMW groups of degree $(m,n)$ as $m,n\to \infty$. 
Addressing this question, we give the following result:

\begin{theoremA}\label{main thm: counting}
There exist $0<\alpha<\beta$ such that, for all sufficiently large $m$ and~$n$,
\[ (mn)^{\alpha mn} \le |\BMW_{\text{vs}}(m,n)/\comm|\le |\BMW(m,n)/\comm| \le (mn)^{\beta mn}. \]
where $\BMW_{\text{vs}}(m,n)$ is the collection of BMW groups, up to conjugacy, of degree $(m,n)$ that contain an index 4 simple subgroup.
\end{theoremA}
All BMW groups contain an index 4 normal subgroup, so the index in the above theorem is as small as possible.

\begin{remark}\label{rem: counting hyperbolic manifolds}
    Compare the above result with the bounds obtained by \cite{burger2002counting,gelander2014counting}: there exist $0<\alpha'<\beta'$ such that the number of commensurability classes of hyperbolic manifolds of volume at most $v$ is bounded between $v^{\alpha' v}$ and $v^{\beta' v}$.
\end{remark}

\subsection*{A random model for irreducible BMW groups.} 
The random model we define is based on a combinatorial description of BMW groups (more precisely, of involutive BMW groups) given in \S\ref{sec:structure_sets}. We postpone the definition of the model to \S\ref{sec: random model}, and only highlight its main properties in Theorem \ref{main thm: random model} below. This model does not capture all possible BMW groups. It was chosen predominantly for its relative ease of computations on the one hand, and its naturality on the other.

A BMW group is \emph{irreducible} if it does not  contain a subgroup of finite index that is isomorphic to the direct product of two free groups. 
A group is \emph{just-infinite} if it is infinite and has only finite proper quotients. It is \emph{hereditarily just-infinite} if all its finite-index subgroups are just-infinite.

\begin{theoremA}\label{main thm: random model}
A random BMW involution group of degree $(m,n)$, with $n>m^5$,
is  hereditarily just-infinite and, in particular, is irreducible 
with probability at least $1-\frac{C}{m}$, 
where $C$ is a constant that is independent of $m$ and $n$.
\end{theoremA}

\begin{remark} In fact, when $n > m^5$ we have that:
\begin{itemize}
    \item 
    An arbitrary
    BMW group of degree $(m,n)$ is irreducible if and only if it has non-discrete projections to both its factors. 
    Something stronger happens in the random model: asymptotically almost surely, the projection to $\Aut(T_m)$ (resp. to $\Aut(T_n)$) of a random BMW group in the model contains the universal groups $U(A_m)$ (resp. $U(A_n)$).
    \item By the previous remark and a rigidity theorem for BMW groups \cite[Theorem 1.4.1]{bmz09lattices} (see also Theorem~\ref{thm:bmzrigiidity}), asymptotically almost surely, two random BMW groups are not isomorphic. 
\end{itemize}
\end{remark}
We give the following conjecture regarding this random model:
\begin{conjecture}
In the above range of $m$ and $n$, 
a random BMW group 
is asymptotically almost surely not residually finite and  consequently is virtually simple by Theorem~\ref{main thm: random model}.
\end{conjecture}
More generally, we conjecture the following:
\begin{conjecture}\label{conj:vsimple_most}
As $m,n\to\infty$, the proportion of virtually simple  BMW groups of degree $(m,n)$,
up to conjugacy,
tends to 1.
\end{conjecture}
Positive evidence for Conjecture \ref{conj:vsimple_most} has been given by Rattagi \cite{rattaggi2004computations} and Radu \cite{radu20newlattices} 
for small values of $m$ and $n$.

\subsection*{Outline} In \S\ref{sec:involutive_BMWs} we discuss involutive BMW groups and a combinatorial description of them. In \S\ref{sec:bounds}, we bound the number of BMWs from above, 
and in \S\ref{sec:counting_commensurability}, we prove the more difficult lower bound, giving Theorem \ref{main thm: counting}.
In \S\ref{sec:random_model} we present our random model for BMW groups. Next, in \S\ref{sec:atree_local} and \S\ref{sec:btree_local}, we show that the local actions are alternating or symmetric with high probability.
Finally, in \S\ref{sec:irreducibility} we prove Theorem~\ref{main thm: random model}.
We note that sections \S\ref{sec:bounds} and \S\ref{sec:counting_commensurability} can be read independently of \S\ref{sec:random_model}, \S\ref{sec:atree_local}, \S\ref{sec:btree_local} and \S\ref{sec:irreducibility} (and vice versa).

\subsubsection*{Acknowledgements.} We thank Pierre-Emmanuel Caprace for his comments on the manuscript and for his suggestion to use the main result of his paper \cite{caprace2018radius} to improve the bound in Theorem \ref{main thm: random model}.

\section{Involutive BMW Groups}\label{sec:involutive_BMWs}

 An \emph{involutive} BMW group $\Gamma$ of degree $(m,n)$ is a BMW group such that for every edge $e$ of $T_m\times T_n$, there is some $g\in \Gamma$ that interchanges the endpoints of $e$. Since $\Gamma$ acts simply transitively on vertices, such an element $g$ must be an involution.
 We let $\iBMW(m,n)$
 denote the set of all involutive BMW groups of degree $(m,n)$, up to conjugacy in $\Aut(T_m)\times \Aut(T_n)$. 
 A BMW group of degree $(m,n)$ is irreducible if and only if  the projection of $\Gamma$ to either $\Aut(T_n)$ or $\Aut(T_m)$ is not discrete  \cite[Proposition 1.2]{burgermozes2000lattices}.
 
Any tree $T$ is bipartite; let $\Aut^+(T)$ be the index 2 subgroup of $\Aut(T)$ preserving the bi-partition of $T$.
Given a BMW group $\Gamma$ of degree $(m,n)$, denote by  $\Gamma^+ = \Gamma \cap (\Aut^+(T_m)\times \Aut^+(T_n))$ the index 4 subgroup of $\Gamma$ preserving the bipartitions of $T_m$ and $T_n$. This subgroup is always torsion-free \cite[Lemma 3.1]{radu20newlattices}.

\subsection{Structure sets}\label{sec:structure_sets}
In this section we describe structure sets which encode presentations for BMW groups.
For the rest of this article, we  fix countable indexing sets $a_1, a_2, \dots$ and $b_1, b_2, \dots$, and for each $k \in \bbN$, we set $A_k :=\{a_1,\dots, a_k\}$ and $B_k :=\{b_1,\dots, b_k\}$.

\begin{definition}[Structure set]\label{def: structure set}
An \emph{$(m,n)$-structure set} $S$ is a collection of subsets of $A_m\sqcup B_n$ such that:
\begin{enumerate}
    \item each element of $S$ is of the form $\{a_i,b_k,a_j,b_l\}$ where $a_i,a_j\in A_m$ and $b_k,b_l\in B_n$, and
    \item for every $a\in A_m$ and $b\in B_n$, $\{a,b\}$ is a subset of exactly one set in $S$.
\end{enumerate} 
Let $\calS_{m,n}$ denote the set of all $(m,n)$-structure sets.

For a structure set $S$, denote by $R_S$ the set of \emph{words in $A_m\sqcup B_n$} defined as
\[R_S = \{ a_i b_k a_j b_l \;|\;\{a_i,b_k,a_j,b_l\} \in S\} .\]
\end{definition}

\begin{remark}\label{rmk: structure sets}
    In the definition of a structure set, the elements $a_i,a_j$ (and similarly $b_k,b_l$) of a set $\{a_i,b_k,a_j,b_l\}\in S$ are not assumed to be distinct, so some $\{a_i,b_k,a_j,b_l\}\in S$ may have fewer than 4 elements. We often still write repeating elements in these subsets -- e.g. $\{a_i,b_k,a_i,b_l\}$ instead of $\{a_i,b_k,b_l\}$. In this example, the word $a_i b_k a_i b_l$ is one of the words in $R_S$ corresponding to $\{a_i, b_k, a_i, b_l\} = \{a_i, b_k, b_l\}$.
\end{remark}
\begin{remark} \label{rmk: bipartite graphs}
 A useful point of view on structure sets is given by partitions of the complete bi-partite graph:
    Let $K_{m,n}$ be the complete bi-partite graph on $A_m \sqcup B_n$. Given  $\{a_i, b_k, a_j, b_l\}\in S$, one can assign to it the closed (possibly degenerate) path of length 4 in $K_{m,n}$ connecting the vertices $a_i,a_j$ to the vertices $b_k,b_l$. In this way, one can think of an $(m,n)$-structure set as a partition of the edges of the complete bi-partite graph on the vertices $A_m \sqcup B_n$ into closed paths of length 4 such that each edge belongs to exactly one such path.
\end{remark}

A (combinatorial) \emph{square complex} is a 2-complex in which 2-cells (squares) are attached along combinatorial paths of length four. A \emph{VH-complex} is a square complex in which the set of 1-cells (edges) is partitioned into vertical and horizontal edges such that the attaching map of each square alternates between them. 
We regard the product $T_m\times T_n$ of trees as a VH-complex where an edge is  horizontal if it lies in $T_m\times \{v\}$ for some $v\in T_n$, and vertical  if it lies in $\{v\}\times {T_n}$ for some $v\in T_m$.

\begin{definition}[Marking] \label{def:marking}
    A \emph{marking} $\mathcal M$  on  $T_m\times T_n$ is a choice of a base vertex $o \in T_m\times T_n$ and an identification of the horizontal (resp. vertical) edges incident to $o$ with $A_m$ (resp. $B_n$).   
    An element $g \in \Aut(T_m) \times \Aut(T_n)$ is said to \emph{fix $\mathcal M$} if $g$ fixes $o$ and fixes all edges adjacent to $o$. 
\end{definition}

Fix a marking $\mathcal M$ on $T_m \times T_n$ with base vertex $o$.
Let $\BMW_{\mathcal M}(m,n)$ be the set of all involutive BMW groups of degree $(m,n)$, up to an automorphism fixing $\mathcal M$.
In other words, two BMW groups $\Gamma$ and $\Gamma'$ of degree $(m,n)$ are equal in  $\BMW_{\mathcal M}(m,n)$ if and only if $\Gamma = g \Gamma' g^{-1}$ where $g$ is an element of $\Aut(T_m) \times \Aut(T_n)$ that fixes $\mathcal M$.
Let $\calS_{m,n}$ be the set of all $(m,n)$-structure sets. 
We now describe how to obtain a bijection:
\[\Phi: \BMW_{\mathcal M}(m,n) \to \calS_{m,n}\]

Let $\Gamma$ be an involutive BMW group of degree $(m,n)$. 
As $\Gamma$ is involutive, each edge of $T_m \times T_n$ is stabilized by a unique element of $\Gamma$.
By a slight abuse of notation, we let $a_i$ (resp. $b_i$) denote the element of $\Gamma$ that stabilizes the edge that is adjacent to $o$ with label $a_i$ (resp. $b_i$).
As $\Gamma$ acts freely and transitively on the vertices of $T_m \times T_n$, its action induces a well-defined, $\Gamma$-invariant labeling of the edges of $T_m \times T_n$ which we generally call the \emph{$\Gamma$-induced labeling}.
Moreover, it is readily checked that the $1$--skeleton of $T_m \times T_n$ with this labeling is the Cayley graph for $\Gamma$ with generating set $\{a_1, \dots, a_m, b_1, \dots, b_n \}$ (where bigons in this Cayley graph are collapsed to edges).

We now describe how to form an $(m,n)$-structure set $S$ associated to $\Gamma$. 
Let $S$ be the collection of subsets $\{a_i,b_k,a_j,b_l\}$ such that there exists a square in $T_m \times T_n$ whose edges are labeled $a_i,b_k,a_j,b_l$ with respect to the $\Gamma$-induced labeling.
Note that since $\Gamma$ acts simply transitively on the vertices of $T_m\times T_n$ and preserves the $\Gamma$-induced labeling, it suffices to only consider the squares 
containing $o$.

To show that $S$ is a structure set, let $a \in A_m$ and $b \in B_n$. 
There exists a unique square $s$ which contains both edges incident to $o$ labeled by $a$ and $b$. 
Since $\Gamma$ acts simply transitively on vertices, any other square which contains two edges labeled by $a$ and $b$ is in the orbit  of $s$ and consequently its edges 
have the same labels as $s$. Thus, there is a unique $\{a_i,b_k,a_j,b_l\}\in S$ containing $a$ and $b$.
So, $S$ is indeed an $(m,n)$-structure set. 
We say that $S$ is the structure set \emph{associated with $\Gamma$}, and we define $\Phi([\Gamma]) = S$, where $[\Gamma]$ is the equivalence class in $\BMW_{\mathcal M}(m,n)$ containing~$\Gamma$. 

Additionally, we conclude that $\Gamma$ has the presentation
\[\langle A_m\sqcup B_n\mid \{ a^2 \mid a\in A_m\} \cup \{b^2  \mid b\in B_n\} \cup R_S \rangle. \label{eqn:presentation} \]
This follows since 
we can take the $1$--skeleton of $T_m \times T_n$, label it with the $\Gamma$-induced labeling (i.e., form the Cayley graph for $\Gamma$) and attach $2$--cells corresponding to the relations $R_S$. The resulting complex can also be obtained from the Cayley complex for $\Gamma$ by collapsing each bigon corresponding to the relations $a^2$ and $b^2$ to an edge. In fact, it is just $T_m \times T_n$ with the $\Gamma$-induced edge labeling.

We now need to check that $\Phi$ is well defined. Suppose that $\Gamma'$ is an involutive BMW group of degree $(m,n)$ that is conjugate to $\Gamma$ by some $g \in \Aut(T_m) \times \Aut(T_n)$ that fixes $\mathcal M$. 
Then, as $g$ fixes $\mathcal M$, the $\Gamma$-induced labeling and $\Gamma'$-induced labeling of $T_m \times T_n$ agree on all squares that contain $o$. 
It follows by construction that the structure sets associated to $\Gamma$ and $\Gamma'$ are equal. Consequently, $\Phi$ is well-defined.

We now check that $\Phi$ is injective.
Suppose that $\Gamma$ and $\Gamma'$ are involutive BMW groups of degree $(m,n)$ and that $\Phi([\Gamma]) = \Phi([\Gamma'])$. 
Consequently, the $\Gamma$-induced labeling and the $\Gamma'$-induced labeling of $T_m \times T_n$ agree on all squares which contain $o$.
It now readily follows that $\Gamma'$ is conjugate to $\Gamma$ by some element $g \in \Aut(T_m) \times \Aut(T_n)$ that fixes the marking $\mathcal M$.
Thus, $\Gamma$ and $\Gamma'$ are equal in $\BMW_{\mathcal M}(m,n)$.

Finally, we check that $\Phi$ is surjective.
Let $S$ be an $(m,n)$-structure set.
Then the group $\Gamma$ with presentation as in (\ref{eqn:presentation}) is an involutive BMW group (see for instance \cite{Caprace-survey}).
Moreover, $\Phi([\Gamma]) = S$.
We have thus shown:
\begin{proposition} \label{prop:structure_set_bmw_bijection}
    Let $\mathcal M$ be a marking of $T_m \times T_n$.
    There is a bijection 
    \[\Phi: \BMW_{\mathcal M}(m,n) \to \calS_{m,n}.\]
    Moreover, for each $S \in \calS_{m,n}$,
    each representative of $\Phi^{-1}(S)$ has the presentation
    \[\langle A_m\sqcup B_n\mid \{ a^2 \mid a\in A_m\} \cup \{b^2  \mid b\in B_n\} \cup R_S\rangle \]
\end{proposition}

Let $g \in  \Aut(T_m) \times \Aut(T_n)$ be an automorphism.
We describe how $g$ acts on markings.
Let $\mathcal M$ be a marking of $T_m \times T_n$ with base vertex $o$.
Then $g$ induces a new marking $\mathcal M' = g\mathcal M$ whose base vertex is $o' = go$ and such that the label of an edge $e$ adjacent to $o'$ is equal to the label of $g^{-1}e$ under the marking $\mathcal M$.
This action of $g$ also induces a bijection $\Psi_g: \BMW_{\mathcal M}(m,n) \to \BMW_{\mathcal M'}(m,n)$ by sending $[\Gamma] \in \BMW_{\mathcal M}(m,n)$ to $[g\Gamma g^{-1}] \in \BMW_{\mathcal M'}(m,n)$. 

Let $S$ be an $(m,n)$ structure set, and let  $\mu \in \Sym(A_m)$ and $\nu \in \Sym(B_m)$ be permutations.
We can form a new $(m,n)$-structure set $S'$ by applying the permutation $(\mu, \nu) \in \Sym(A_m)\times \Sym(B_n) \le \Sym(A_m \sqcup B_n)$ to the subsets in $S$.
We say that $S'$ is a \emph{relabeling of $P$ induced by $\mu$ and $\nu$}.

Let $\Gamma' = g\Gamma g^{-1}$ for some $g\in \Aut(T_m) \times \Aut(T_n)$. Since $\Gamma$ acts vertex transitively, we may assume, without loss of generality, that $g$ fixes $o$. Thus, $g$ induces  permutations $\mu\in\Sym(A_m)$ and $\nu\in \Sym(B_m)$ on the labels (in the marking $\mathcal M$) of the edges incident to $o$. It readily follows that the structure set of $S'$ of $\Gamma'$ is obtained from the structure set $S$ of $\Gamma$ by relabeling.

Thus, if $\Gamma$ and $\Gamma'$ are conjugate BMW groups, then their associated structure sets are the same up to a relabeling, regardless of a choice of marking.
Conversely, suppose that $S'$ is an $(m,n)$-structure set that is a relabeling of a structure set $S$ induced by $\mu \in \Sym(A_m)$ and $\nu \in \Sym(B_n)$.
Then we can choose a marking $\mathcal M$ on $T_m \times T_n$ and a BMW group $\Gamma$ whose associated structure set is $P$. 
Additionally, we can choose a $g \in \Aut_{\calM}(T_m \times T_n)$ so that the induced action of $g$ on the labels of the horizontal and vertical edges around $o$ is given by $\mu$ and $\nu$ respectively. 
From this, we have that $g \Gamma g^{-1}$ is a BMW group conjugate to $\Gamma$ whose structure set is $S'$.
We have thus shown the following:
\begin{proposition}
    There is a bijection 
    \[\Psi: \BMW(m,n) \to \calS_{m,n}/\text{relabeling}.\]
\end{proposition}

\subsection{Local actions}
Let $X$ be a locally finite graph. For every vertex $v\in V(X)$, let $E(v)$ be the set of edges of $X$ incident to $v$. 
If a group $\Gamma$ acts  on  $X$, the \emph{local action} of $\Gamma$ on $X$ at the vertex $v$ is the induced action of $\Stab_\Gamma(v)$ on the set $E(v)$.  
By abuse of terminology, we will also refer to the image of the action $\Stab_\Gamma(v)\to \Sym(E(v))\cong \Sym(n)$ as the local action, where $n=\lvert E(v)\rvert$.
The \emph{local actions of $\Gamma\leq \Aut(T_m)\times \Aut(T_n)$} are the local actions of $\Gamma$ on $T_m$ and $T_n$. More specifically, we call the local action of $\Gamma$ on $T_m$ the \emph{$A$-tree local action},  and the local action of $\Gamma$ on $T_n$ the \emph{$B$-tree local action}.

We show how to read off the local action of an involutive BMW group of degree $(m,n)$ from the corresponding structure set.
First note that since a BMW group of degree $(m,n)$ acts transitively on the vertices of $T_m\times T_n$, its local actions on $T_m$ (resp. $T_n$) at different vertices are conjugate actions. We can thus refer to \emph{the} local action on $T_m$ (resp. on $T_n$) as the conjugacy class of the local action at some vertex of $T_m$ (resp. $T_n$).

Let us focus on the local action on $T_n$.
Let $\calM$ be a marking of $T_m\times T_n$ 
with base vertex $o$.
Let $\pi_A$ and $\pi_B$ be the projections of $T_m\times T_n$ to the first and second factors respectively.
Let $o_B = \pi_B(o) \in T_n$. 
Edges incident to $o_B$ in $T_n$ 
are
%can be
labeled by elements of $B_n$ as follows: the label of an edge $e$ incident to $o_B$ is the label of the unique edge $e'$ of $T_m\times T_n$ incident to $o$ such that $\pi_B(e') = e$.

Let $\Gamma$ be an involutive BMW group of degree $(m,n)$ with structure set $S$. 
The local action of $\Gamma$ at the vertex $o_B$ can be identified with a subgroup of $\Sym(A_m)\simeq \Sym(m)$.
Recall that $\Gamma$ is generated by the elements~$a_1,\ldots,a_m$, $b_1,\ldots,b_n$.
Observe that  $\Stab_{\Gamma}(o_B)=\gen{a_1,\ldots,a_m}$, and thus the local action on $T_n$ is generated by the action of each one of $a_1,\ldots,a_m\in A_m$.
Denote by $\alpha_i \in \Sym(n)\simeq \Sym(B_n)$ the action of $a_i$ on $B_n$. 

\begin{lemma}\label{lem: from structure set to local actions}
Let $1\leq i\leq m$ and $1\le k,l \le n$. Then  $\alpha_i(k)=l$ if and only if $\{a_i,b_k,a_j,b_l\}\in S$ for some $a_j\in A_m$.
\end{lemma}
\begin{proof}
Fix $1\leq i\leq m$ and $1\le k \le n$. Then, by the definition of a structure set, there exists a unique $\{a_i,b_k,a_j,b_l\} \in S$ containing both $a_i$ and $b_k$. 
Thus, to prove the claim, it is enough to show that $\alpha_i(k)=l$.
Let $e$ be the edge of $T_m$ that is labeled by $b_k$ and incident to $o_B$. 
To show that $\alpha_i(k)=l$, we need to show that the label of $a_i e$ is $b_l$.

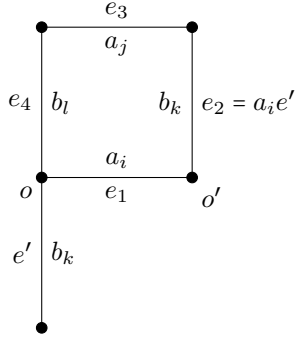
\begin{figure}[h]
	\begin{center}
		\begin{tikzpicture}
\node  at (1,2)  [above] {$e_3$};
\node  at (1,2)  [below] {$a_j$};

\node  at (1,0)  [above] {$a_{i}$};
\node  at (1,0)  [below] {$e_1$};

\node  at (0,1)  [left] {$e_4$};
\node  at (0,1)  [right] {$b_l$};

\node  at (2,1)  [right] {$e_2=a_ie'$};
\node  at (2,1)  [left] {$b_k$};

\node at (2,0) [below right] {$o'$}; 
\node at (0,0) [below left] {$o$}; 

\foreach \x in {0,2}{
	\foreach \y in {0,2}{
\filldraw [black] (\x,\y) circle (2pt);}}

\foreach \x in {0,2}{
\draw (0,\x) -- (2,\x);
}

\foreach \x in {0,2}{
\draw (\x,0) --(\x,2);
}

\draw (0,0) -- (0,-2);
\node  at (0,-1)  [left] {$e'$};
\node  at (0,-1)  [right] {$b_k$};
\filldraw [black] (0,-2) circle (2pt);
\end{tikzpicture}
	\end{center}
	\caption{Determining the  action  of $\alpha_i(k)$.}\label{fig:square}
\end{figure}

There exists a unique edge $e'\in T_m\times T_n$ labeled $b_k$ and incident to $o$. 
The element $a_i$ acts on $T_m\times T_n$ by mapping $o$ to the endpoint $o'$ of the unique edge $e_1$ labeled $a_i$ incident to $o$.
Since $\Gamma$ preserves the $\Gamma$-induced labeling, $a_ie'$ is the unique edge $e_2$ of $T_m\times T_n$ incident to $o'$ labeled $b_k$. The edges $e_1,e_2$ are adjacent edges of a (unique) square in $T_m\times T_n$.
Let $e_1,e_2,e_3,e_4$ be the edges of this square as shown in Figure \ref{fig:square}. 
By the definition of $S$, their respective $\Gamma$-induced labels are $a_i,b_k,a_j,b_l$.
We get that $a_i e=\pi_B(a_i e')=\pi_B(e_2)=\pi_B(e_4)$.
Since $e_4$ is incident to $o$, the label of $a_i e$ is the same as that of $e_4$, namely $b_l$.
\end{proof}

We call the involutions $\alpha_1,\ldots,\alpha_m \in \Sym(n)$ the \emph{$B$-tree local involutions}. 
Similarly, we can define the \emph{$A$-tree local involutions} $\beta_1,\ldots,\beta_n\in \Sym(m)$ corresponding to the local actions of $b_1,\ldots,b_n$ on the tree $T_m$.

\subsection{Virtual simplicity of BMW groups}
The following theorem is a corollary of \cite[Propositions 3.3.1, 3.3.2]{burgermozes2000localtoglobal} and \cite[Theorem 4.1]{burgermozes2000lattices}.

\begin{theorem}[Burger-Mozes]\label{thm: BM just infinite}
    Let $m,n\ge 6$, let $\Gamma$ be an irreducible BMW group of degree $(m,n)$, and assume that the local actions of $\Gamma$ on $T_m$ and $T_n$ contain the alternating groups $\Alt(m)$ and $\Alt(n)$ respectively. Then $\Gamma$ is hereditarily just-infinite.
\end{theorem}

For a group $\Gamma$, its \emph{finite residual} $\FR{\Gamma}$ is the intersection of all finite-index subgroups of $\Gamma$.
The following widely known lemma 
is a tool to prove virtual simplicity of a group.

\begin{lemma}\label{lem: simple finite residual}
    If 
    a group 
    $\Gamma$ is hereditarily just-infinite and is not residually finite, then $\FR{\Gamma}$ is a finite index, simple subgroup of $\Gamma$.
\end{lemma}
\begin{proof}
 Since $\Gamma$ is not residually finite, the finite residual $\FR{\Gamma}$ is a non-trivial normal subgroup of $\Gamma$. By assumption $\Gamma$ is just-infinite, thus $\FR{\Gamma}$ must have finite index in $\Gamma$. As $\Gamma$ is hereditary just infinite, $\FR{\Gamma}$ is itself just-infinite, and thus cannot contain any non-trivial infinite-index normal subgroup. On the other hand, by definition and since $\FR{\Gamma}$ has finite index in $\Gamma$, $\FR{\Gamma}$ cannot have any non-trivial finite-index normal subgroup. Thus, $\FR{\Gamma}$ is simple.
\end{proof}

\section{Upper bounds on BMW counts}\label{sec:bounds}
In this section we give upper bounds for the number of conjugacy classes of involutive BMW groups.
We then use a result of Burger-Mozes-Zimmer to bound the number of BMW groups that are abstractly commensurable to a given BMW group with primitive local actions and simple type-preserving subgroup.

\subsection{Upper bound on conjugacy classes of involutive BMWs}

\begin{proposition} \label{prop:bmw_upper_count}
    There are at most $(mn)^{mn}$ conjugacy classes of involutive BMW groups of degree $(m,n)$. 
\end{proposition}
\begin{proof}
    Recall that $\calS_{m,n}$ denotes the set of $(m,n)$-structure sets.
    Every $S\in \calS_{m,n}$ defines a function $f_S :A_m \times B_n \to A_m \times B_n$ by $f_S(a,b) = (a',b')$ if $\{a,b,a',b'\}\in S$. 
    This function is well-defined by the definition of a structure set. 
    We can reconstruct $S$ from $f_S$ by $$S=\{\{a,b,a',b'\}\mid (a',b')=f_S(a,b), a\in A_m, b\in B_n\}.$$ 
    Thus we get an injective map $\calS_{m,n} \inj (A_m \times B_n )^{ A_m \times B_n}$ mapping $S\mapsto f_S$. Consequently:
    $$|\BMW(m,n)| \le |\BMW_{\mathcal M}(m,n)|=|\calS_{m,n}| \le|(A_m \times B_n )^{ A_m \times B_n}| = (mn)^{mn}$$
    where the equality $|\BMW_{\mathcal M}(m,n)|=|\calS_{m,n}|$ follows from the bijection in Proposition~\ref{prop:structure_set_bmw_bijection}.\qedhere
\end{proof}

\begin{remark}
For general (not necessarily involutive) BMW groups, using the $(m,n)$-datum defined in \cite{radu20newlattices} (which is analogous to structure sets defined here), a similar proof gives a number $\beta>0$ such that the number of conjugacy classes of BMW groups is bounded by $(mn)^{\beta mn}$.
\end{remark}

\subsection{$(m,n)$-complexes and type-preserving subgroups}
In this subsection we associate to each involutive BMW-group a certain edge-labeled square complex that completely describes the group.
These complexes will allow us to deduce a count on the number of involutive BMW groups with the same type-preserving subgroup up to conjugation.

\begin{definition} \label{def:complex}
	An \emph{$(m,n)$-complex} is a 
	VH-complex $Y$ such that
	\begin{enumerate}
		\item $Y$ has exactly 4 vertices: $v_{00}$, $v_{10}$, $v_{01}$ and $v_{11}$.
 		\item There are exactly $m$ edges between $v_{00}$ and $v_{10}$ (resp. $v_{01}$ and $v_{11}$), all of which are horizontal and labelled by distinct elements of $A_m$.
 		\item There are exactly  $n$ edges between $v_{00}$ and $v_{01}$ (resp. $v_{10}$ and $v_{11}$), all of which are vertical and labelled by distinct elements of $B_n$.
		\item For each horizontal edge $e_1$ and vertical edge $e_2$ of $Y$, there is a unique  square containing both $e_1$ and $e_2$. 
		\item There is a label-preserving  vertex-transitive  action on $Y$.
	\end{enumerate} 
\end{definition}

\begin{remark}
    By (4) above, an $(m,n)$-complex contains exactly $mn$ squares.  We also note that the label preserving automorphism group of an $(m,n)$-complex is precisely $\mathbb{Z}_2\times \mathbb{Z}_2$.
\end{remark}

\begin{lemma} \label{lem:complex_bmw_correspondence}
    Fix a marking on $T_m \times T_n$, and let $\Gamma$ be an involutive BMW group of degree $(m,n)$. Label the edges of $T_m \times T_n$ by the $\Gamma$-induced labeling.
    Then $\Gamma^+\backslash (T_m\times T_n)$ is an $(m,n)$-complex.
    
    Conversely, let $Y$ be an $(m,n)$-complex. Then the set $S$, consisting of the subsets $\{a_i,b_k,a_j,b_l\}$ where $a_i,b_k,a_j,b_l$ are the labels of the edges of squares in $Y$, is an $(m,n)$-structure set.
\end{lemma}
\begin{proof}
    Let $\Gamma$ be as in the statement of the lemma. 
    The  type-preserving subgroup $\Gamma^+ < \Gamma$ acts freely, so we can consider the quotient  complex $Z \coloneqq \Gamma^+\backslash (T_m\times T_n)$. 
    As edge labels pass to the quotient and as $\Gamma/\Gamma^+\simeq \bbZ_2\times \bbZ_2$ acts transitively on the vertices of $Z$, $Z$ is an $(m,n)$-complex as required. 
    The proof of the converse statement follows from Definitions~\ref{def: structure set}~and~\ref{def:complex}.
\end{proof}

We will need the following lemma counting the number of possible $(m,n)$-complexes with isomorphic square complexes.
\begin{lemma}\label{lem:square_iso}
    Given any $(m,n)$-complex $C$, there are at most $2(n!m!)^2$ distinct $(m,n)$-complexes that are isomorphic to $C$ as unlabeled square complexes (i.e. isomorphic via a cellular isomorphism that does not necessarily preserve labels or the VH-structure).
\end{lemma}
\begin{proof}
	Let $Y$ be an the underlying square complex of an $(m,n)$-complex $C$. There are 
	at most
	two ways of giving $Y$ a 
	suitable 
	$VH$-structure.
	After choosing a $VH$-structure,  $Y$ has 4 vertices, $2m$ horizontal edges and $2n$ vertical edges, and there are at most $(n!m!)^2$ ways to choose labels to obtain an $(m,n)$-complex. Thus there are at most $2(n!m!)^2$ possible $(m,n)$-complexes that are isomorphic to $Y$ as unlabeled square complexes.
\end{proof}

The next lemma bounds the number of involutive BMW groups with conjugate type-preserving subgroups.

\begin{lemma} \label{lem:equal_type_preserving}
For each $\Gamma \in \BMW(m,n)$,
        there are at most $2(n!m!)^2$ conjugacy classes of $\Lambda \in \BMW(m,n)$ with $\Gamma^+$ conjugate to $\Lambda^+$.
\end{lemma}
\begin{proof}
    Suppose that $\Gamma^+ = g \Lambda^+ g^{-1}$ for some $g \in \Aut(T_m) \times \Aut(T_n)$.
    We then have that $Y_1 := \Gamma^+ \backslash (T_m \times T_n)$ is isomorphic to $Y_2 := g\Lambda^+g^{-1} \backslash (T_m \times T_n)$ as unlabeled square complexes.
    By Lemma~\ref{lem:square_iso}, $Y_2$ can be one of at most $2(n!m!)^2$ possible $(m,n)$-complexes.
    By Lemma~\ref{lem:complex_bmw_correspondence} and Proposition~\ref{prop:structure_set_bmw_bijection}, such an $(m,n)$-complex completely determines $g \Lambda g^{-1}$ up to a conjugation.
    The lemma now follows.
\end{proof}

Recall that a subgroup $F\leq \Sym(n)$ is \emph{primitive} if no non-trivial partitions of $\{1,\dots,n\}$ is stabilised by $F$.
The following theorem is a reformulation  of a result of Burger--Mozes--Zimmer, which builds on a superrigidity theorem of Monod--Shalom \cite{monodshalom2004cocycle}:
\begin{theorem}[{\cite[Theorem 1.4.1]{bmz09lattices}}]\label{thm:bmzrigiidity}
	Let $G= (\Aut(T_m)\times \Aut(T_n))\rtimes R$, where   $R\leq \bbZ/2\bbZ$  permutes the factors when $m=n$ and is trivial otherwise. Let $\Gamma,\Gamma'\leq \Aut(T_m)\times \Aut(T_n)$ be cocompact lattices with primitive local actions.  Then any isomorphism $\phi:\Gamma \to \Gamma'$ is induced by conjugation in $G$, i.e. there exist some   $g\in G$ such that $ghg^{-1}=\phi(h)$ for all $h\in \Gamma$. 
\end{theorem}
Recall that two groups are \emph{abstractly commensurable} if they have isomorphic finite index subgroups. The previous theorem implies the following:
\begin{proposition}\label{prop: upper bound on commensurable BMWs}
		Let $\Gamma$  be an  involutive BMW  group of degree $(m,n)$ with primitive local actions and with $\Gamma^+$ simple. 
		Then, up to conjugacy, there are at most $2(n!m!)^2$ involutive BMW groups of degree $(m,n)$ that are abstractly commensurable with $\Gamma$.
\end{proposition}
\begin{proof}
	Let $\Lambda$ be an involutive BMW group of degree $(m,n)$ that is abstractly commensurable to $\Gamma$. 
	As $\Gamma^+$ is simple, $\Lambda$  contains a finite index subgroup $H$ that is isomorphic to $\Gamma^+$.
	It follows from Theorem \ref{thm:bmzrigiidity} that $H$ is type-preserving and has four orbits of vertices, so $H=\Lambda^+$.  
	By Theorem~\ref{thm:bmzrigiidity}, $\Gamma^+$ is conjugate to $\Lambda^+$. 
	The result now follows from Lemma~\ref{lem:equal_type_preserving}.
\end{proof}

\section{Counting commensurability classes of BMW groups}\label{sec:counting_commensurability}

In order to count commensurability classes of BMW groups, we first define a partial structure set $S_0$ (see definition below). We  show in Theorem \ref{thm: containing P implies simple of index 4} that if $\Gamma$ is an involutive BMW group whose structure set contains $S_0$, then the index 4 subgroup $\Gamma^+$ is simple. 
We then deduce the lower bound of  Theorem \ref{main thm: counting} by showing that there are sufficiently many such $\Gamma$.

\begin{definition}
    A \emph{partial structure set} $S_0$ is
a collection of subsets of $A_m\sqcup B_n$ of the form $\{a_i,b_k,a_j,b_l\}$, for $a_i,a_j\in A_m$ and $b_k,b_l\in B_n$, such that for every $a\in A_m$ and $b\in B_n$ \emph{at most} one subset of $S_0$ contains both $a$ and $b$. 
\end{definition}

Our starting point is the following involutive BMW group $\Delta$ (denoted as $\Gamma_{4,5,9}$ by Radu \cite{radu20newlattices}).

\begin{theorem*}[{\cite[Theorem 5.5]{radu20newlattices}}]
	The involutive BMW group $\Delta$ of degree $(4,5)$, whose associated $(4,5)$-structure set  is 
	\[S_\Delta\coloneqq \left\{ \begin{split}
		& \{a_1,b_1,a_1,b_1\}, 
		\{a_1,b_2,a_1,b_2\},
		\{a_2,b_3,a_1,b_3\},
		\{a_2,b_1,a_2,b_1\}\\
		& \{a_3,b_2,a_2,b_2\}, 
		\{a_3,b_3,a_3,b_1\}, 
		\{a_1,b_4,a_1,b_4\},
		\{a_4,b_5,a_1,b_5\}\\ & 
		\{a_3,b_5,a_2,b_4\}, 
		\{a_4,b_2,a_4,b_1\}, 
		\{a_4,b_4,a_4,b_3\}
	\end{split}
\right\},\] 
	satisfies $\FR{\Delta} = \Delta^+$, where $\FR{\Delta}$ is the intersection of all finite index subgroups of~$\Delta$.
\end{theorem*}

Let $\alpha^\Delta_{1}, \dots, \alpha^\Delta_{4}$ be the corresponding $B$-tree  local involutions of $S_\Delta$, and let $\beta^\Delta_{1}, \dots, \beta^\Delta_{5}$ be the $A$-tree local involutions.
The $A$-tree and $B$-tree local involutions can be computed explicitly (using Lemma~\ref{lem: from structure set to local actions}) 
to show that $\gen{\alpha^\Delta_1,\ldots,\alpha^\Delta_4} = \Sym(5)$ and $\gen{\beta^\Delta_1,\ldots,\beta^\Delta_5} = \Sym(4)$.
Thus, the local actions of $\Delta$ on $T_4$ and $T_5$ are $\Sym(4)$ and $\Sym(5)$ respectively.
Let $m\ge 13, n\ge 14$ be integers.

For natural numbers $k \le k'$, denote by $\llbracket k,k'\rrbracket:=\{k,k+1,\dots,k'\}$. 
Fix the following three involutions in $\Sym(\llbracket 6, n\rrbracket)$:
\begin{align*}
	\alpha'_1 &= (7 \quad 10) (8 \quad 11) (12 \quad 13) (14 \quad 15) (16 \quad 17)\dots \\
	\alpha'_2 &= (9 \quad 7)(10 \quad 8)(11 \quad 12)(13 \quad 14)(15 \quad 16)\dots\\
	\alpha'_3 &= (6 \quad 9).
\end{align*}
Note that these three involutions generate  $\Sym(\llbracket6,n\rrbracket)$.
Similarly, fix the following three involutions of $\Sym(\llbracket5, m\rrbracket)$:
\begin{align*}
	\beta'_1 &= (6 \quad 9) (7 \quad 10) (11 \quad 12) (13 \quad 14) (15 \quad 16)\dots \\
	\beta'_2 &= (8 \quad 6)(9 \quad 7)(10 \quad 11)(12 \quad 13)(14 \quad 15)\dots\\
	\beta'_3 &= (5 \quad 8).
\end{align*}
These involutions  generate $\Sym(\llbracket5,m\rrbracket)$.

Define the following partial structure sets:
\begin{align*}
	S_{AR} &= \{ \{a_{i},b_k, a_{i}, b_{\alpha_i'(k)}\}~|~ 1 \le i \le 3, ~6 \le k \le n\}\\
	S_{BR} &= \{\{a_{i},b_k,  a_{\beta_k'(i)},b_{k}\}~|~5 \le i \le m,~ 1 \le k \le 3\}\\
	S_{AC}&= \{\{a_4,b_k,a_4,b_k\} ~| ~6\le k\le 8\}\\
	S_{BC}&= \{\{a_i,b_k,a_i,b_k\} ~|~ 5 \le i\le 7,~ 4\le k\le 5\}\\
	S_{AB}&= \{\{a_{i},b_k ,a_{\beta_{k-5}'(i)}, b_{\alpha_{i-4}'(k)}\}~|~ 5 \le i \le 7, ~6 \le k \le 8\}\\
	S_{A} &= \{\{a_{i},b_k ,a_{i}, b_{\alpha_{i-4}'(k)}\}~|~ 5 \le i \le 7, ~9 \le k \le n, ~\alpha_{i-4}'(k)> 8 \}\\
	S_{B} &= \{\{a_i, b_k, a_{\beta_{k-5}'(i)}, b_{k}\}~| ~ 8 \leq i \le m,~ 6 \leq k \le 8, ~\beta_{k-5}'(i)> 7 \}\\
	S_{M} &= \{\{a_{4},b_{n},a_{m},b_{n-1}\},~\{a_{m-2},b_4,a_{m-1},b_{n-2}\}\} \\
	S_{C1} &= \{\{a_i,b_{n},a_i,b_{n}\},~\{a_i,b_{n-1},a_i,b_{n-1}\} ~|~ 8\le i<m\} \\
	S_{C2} &= \{\{a_{m-1},b_k,a_{m-1},b_k\},~ \{a_{m-2},b_k,a_{m-2},b_k\}~|~  9\leq k\leq n-3\}.
\end{align*}
Let 
\[S_0:= S_\Delta \cup S_{AR} \cup S_{BR} \cup S_{AC} \cup S_{BC} \cup S_{AB}\cup S_A \cup S_B \cup S_M \cup S_{C1} \cup S_{C2}.\]
See Figure \ref{fig: partial structure set P} and Table \ref{table:partial_ss_new}.

\begin{figure}
    \centering
    \includegraphics{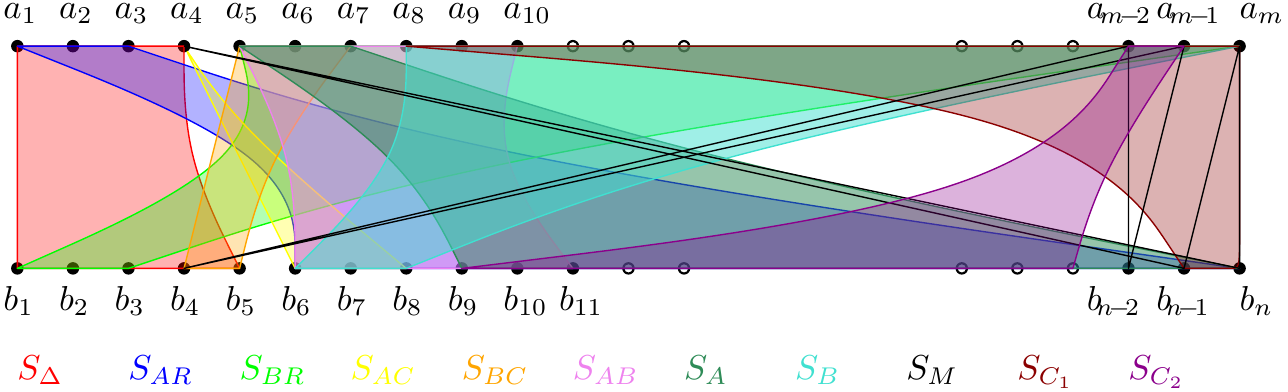}
    \caption{The partial structure set $S_0$ gives rise to a partial partition of the complete bipartite graph $K_{m,n}$ in the sense of Remark \ref{rmk: bipartite graphs}, each set in $S_0$ corresponds to a 4-cycle in $K_{m,n}$ contained in the associated highlighted subset in the figure.}
    \label{fig: partial structure set P}
\end{figure}

\begin{table}[]
	\centering
	\begin{tabular}{l|l|l|l|l|l|l|l}
							&$a_1$--$a_3$& $a_4$	& $a_5$--$a_7$			&$a_8$--$a_{10}$  		&$a_{11}$--$a_{m-3}$    &$a_{m-2}$--$a_{m-1}$	&$a_{m}$\\
		\hline
		$b_1$--$b_3$		&$S_\Delta$	&$S_\Delta$	&$S_{BR}$				&$S_{BR}$				&$S_{BR}$               &$S_{BR}$				&$S_{BR}$\\
		$b_{4}$ 			&$S_\Delta$	&$S_\Delta$	&$S_{BC}$				&						&                       &$S_{M}$				&\\
		$b_{5}$ 			&$S_\Delta$	&$S_\Delta$	&$S_{BC}$ 				&						&                       &						&\\
		$b_6$--$b_8$ 		&$S_{AR}$	&$S_{AC}$	&$S_{AB}$				&$S_{AB}$ \& $S_{B}$	&$S_B$                  &$S_{B}$				& $S_{B}$\\
		$b_9$--$b_{11}$ 	&$S_{AR}$	&			&$S_{AB}$ \& $S_{A}$	&$S_{AB}$				&                       &$S_{C_2}$				&\\
		$b_{12}$--$b_{n-3}$ &$S_{AR}$	&			&$S_{A}$	            &       				&                       &$S_{C_2}$				&\\
		$b_{n-2}$ 			&$S_{AR}$	&			& $S_{A}$				&					    &                    	&$S_{M}$				&\\
		$b_{n-1}$--$b_n$	&$S_{AR}$	&$S_{M}$	& $S_{A}$				&$S_{C_1}$			    &$S_{C_1}$          	&$S_{C_1}$ 				&$S_{M}$\\
	\end{tabular}
	\vspace{.2in}
	\caption{
	Table showing which partial structure sets  possibly
	possess $\{a_i,b_k\}$ as a subset of one its elements. For example, if $\{a_6,b_4\}$ is a subset of some $s\in S_0$, then $s\in S_{BC}$. }\label{table:partial_ss_new}
\end{table}

It is straightforward to check that for any $1 \le i \le m$ and $1 \le k \le n$, $\{a_i,b_k\}$ is a subset of at most one set in $S_0$, making $S_0$ a partial structure set.

\begin{theorem}\label{thm: containing P implies simple of index 4}
    If an $(m,n)$-structure set $S$ contains $S_0$, then its associated involutive BMW group $\Gamma$ satisfies that $\Gamma^+$ is simple.
\end{theorem}

To prove the above theorem, we need to show a few lemmas first.
Throughout, let $S$ be an $(m,n)$-structure set containing $S_0$ and let $\Gamma$ be its associated involutive BMW group.

\begin{lemma}\label{lem:generation}
     The $A$-tree and $B$-tree local actions of $\Gamma$ are $\Sym(m)$ and $\Sym(n)$ respectively.
\end{lemma}
\begin{proof}
    We show the claim for the $B$-tree local action. The argument is similar for the $A$-tree local action.
    Let $\alpha_1, \dots, \alpha_{m}$ be the $B$-tree local involutions of $\Gamma$. 
    By Lemma \ref{lem: from structure set to local actions} and as $S_\Delta,S_{AR} \subset S$, 
    we get that for $1 \le i \le 3$, $\alpha_i = \alpha^\Delta_i \times \alpha'_i$ and that $\alpha_4 = \alpha^\Delta_4 \times \gamma_4$, where $\gamma_4$ is some unknown permutation in $\Sym(\llbracket6, n\rrbracket)$.
    
    Similarly, as  $S_{BR},S_{BC},S_{AB},S_A \subset S$
    we get that for $1 \le i \le 3$, $\alpha_{4 + i} = \id \times \alpha_i'$ where $\id$ is the identity permutation of $\Sym({\llbracket1,5\rrbracket})$.
    As $\alpha_1^\Delta, \dots, \alpha_4^\Delta$ generate $\Sym(\llbracket1,5\rrbracket)$ and $\alpha_1', \dots, \alpha_3'$ generate $\Sym(\llbracket6,n\rrbracket)$,
     $\alpha_1, \dots, \alpha_{7}$ generate a subgroup of $\Sym(n)$ containing 
    \[\Sym(\llbracket1,5\rrbracket) \times \Sym(\llbracket6, n\rrbracket)\]
    Finally, as $S_{BR},S_B,S_M,S_{C1},S_{C2} \subset S$, $\alpha_{m - 1}$ is the transposition $(4, n-2)$. From this we then conclude that $\alpha_1,\ldots,\alpha_m$ generate $\Sym(n)$.
\end{proof}

\begin{lemma}[Finite residual] \label{lem:finite_residual}
If a structure set $S$ contains $S_0$, then its associated involutive BMW group $\Gamma$
satisfies $\FR{\Gamma}=\Gamma^+$.
\end{lemma}

\begin{proof}
Clearly $\FR{\Gamma} \le \Gamma^+$, thus it suffices to show that $\FR{\Gamma}$ has index $4$ in $\Gamma$. More precisely, we show that $\Gamma/\FR{\Gamma} \simeq \bbZ/2 \times \bbZ/2$.

Recall from Proposition \ref{prop:structure_set_bmw_bijection} that $\Gamma$ has a presentation with generators $A_m\sqcup B_n$, so we may identify $A_m$ and $B_n$ with elements of $\Gamma$.
For $g\in \Gamma$ denote by $\bar{g}$ its image in $\Gamma/\FR{\Gamma}$. 
By \cite[Proposition 4.2 (vii)]{Caprace-survey} and as $S_\Delta \subset S$, $\Gamma$ contains the 
subgroup $\Delta$.
By \cite[Theorem 5.5]{radu20newlattices} stated above, $\Delta$ satisfies $\Delta/\FR{\Delta} \simeq \bbZ/2 \times \bbZ/2$. It follows that $\bar{a}_i = \bar{a}_j$ for all $i,j \le 4$ and $\bar{b}_k = \bar{b}_l$ for all $k,l \le 5$. Denote these elements by $\bar{a},\bar{b} \in \Gamma/\FR{\Gamma}$ respectively.

In claim 2 below, we prove that $\bar{b}_i = \bar b$ for all $1 \le i \le n$ and $\bar{a}_i = \bar{a}$ for all $1 \le i \le m$. The lemma follows from this as $\Gamma/\FR{\Gamma}$ is generated by $\bar{a},\bar{b}$ and satisfies the relations $\bar{a}^2=\bar{b}^2=[\bar{a},\bar{b}]=1$. Hence, it is isomorphic to $\bbZ/2\times \bbZ/2$.

We first prove the following claim, showing that certain generators of $\Gamma$ are equal in the quotient.

\medskip 

\textbf{Claim 1: $\bar{a}_i = \bar a_j$ for all $i, j \geq 5$, and $\bar{b}_k = \bar b_l$ for all $k,l \geq 6$}

We show that $\bar{b}_k = \bar b_l$ for all $k,l \geq 6$. The second claim follows from a similar argument.

Let $G$ be the Schreier graph of the action of $\alpha'_1,\alpha'_2,\alpha'_3$ on $\llbracket 6,n\rrbracket$, that is, the graph with vertices $\llbracket 6, n\rrbracket$ and an edge $(k, \alpha'_i(k))$ for every $i \in \{1,2,3\}$ and $k\in \llbracket 6, n\rrbracket$. 
The graph $G$ is connected and not bipartite as otherwise, $\Sym(\llbracket 6,n\rrbracket)=\gen{\alpha'_1,\alpha'_2,\alpha'_3}$ would have an imprimitive action on $\llbracket 6,n\rrbracket$, giving a contradiction.

Let $(k,l)$ be an edge of $G$. We have that $\alpha'_i(k) = l$ for some $i \in \{1,2,3\}$.
Consequently, $\{a_i,b_k,a_i,b_l\}\in S_{AR}$ and $a_i b_k a_i b_l \in R_S$ is a relation in $\Gamma$. 
In particular, $b_l = a_i b_k a_i$ (as generators are involutions).
Since $G$ is not bipartite and is connected, any two vertices of $G$ are connected by an even length path.
It follows that for \textit{any} $k$ and $l$ such that $6\le  k < l \le n$,
there is an even number $p$ such that $\bar b_l = \bar a^p \bar b_k \bar a^p$. 
As $\bar{a}^2=1$, we deduce that $\bar{b}_k = \bar b_l$.
The claim follows.

\medskip 

\textbf{Claim 2: $\bar a = \bar a_i$ for all $1 \le i \le m$ 
, and $\bar b = \bar b_k$ for all $1 \le k \le n$}

By Claim 1, we can define $\bar a' := \bar a_i$ for all $i\geq 5$ and $\bar b' := \bar b_k$ for all $k\geq 6$. 
To prove Claim 2, we need to show that $\bar a = \bar a'$ and $\bar b = \bar b'$.
We show $\bar a = \bar a'$. The second statement follows from a similar argument.

First note that since $S_{AR}\subset S$ the word $a_1 b_{k} a_1 b_{\alpha_1'(k)}\in R_S$ is a relation in $\Gamma$ for some $6\leq k, \alpha_1'(k) < n - 2$.
In the quotient, this relation becomes $\bar{a}\bar{b}'\bar{a}\bar{b}'=1$ which implies that $\bar a$ commutes with $\bar b'$.
Next, since $S_M \subset S$ the word $a_{4}b_{n}a_{m}b_{n-1}\in R_S$ is a relation in $\Gamma$. In the quotient, this gives $\bar a \bar b' \bar a' \bar b'=1$.
As $\bar b'$ commutes with $\bar a$, we see that $\bar a = \bar a'$. The claim follows.
\end{proof}

We are now ready to prove Theorem~\ref{thm: containing P implies simple of index 4}.
\begin{proof}[Proof of Theorem \ref{thm: containing P implies simple of index 4}]
By Lemma~\ref{lem:generation}, the local actions of the group $\Gamma$ are the full symmetric groups on $m$ and $n$ elements. 
Moreover, $\Gamma$ is irreducible and not residually finite as it contains $\Delta$ (see \cite[Proposition 4.2 vii)]{Caprace-survey}).
The theorem then follows from Theorem \ref{thm: BM just infinite}, Lemma \ref{lem: simple finite residual} and Lemma~\ref{lem:finite_residual}.
\end{proof}

\begin{lemma}\label{lem: lower bound on labelled involutive BMW}
There exists a number $\alpha>0$ such that, for all integers $m >0$ and $n>0$, the number of $(m,n)$-structure sets is at least $ (mn)^{\alpha mn}$.
\end{lemma}

\begin{proof}
Without loss of generality assume that $m\le n$. Let $\calI_n \subseteq \Sym(n)$ be the subset of all involutions. For any $m$ involutions, $\alpha_1,\ldots,\alpha_m \in \calI_n$ we can define a structure set $$S = \{\{a_i, b_k, a_i, b_{\alpha_i(k)}\}\mid  1\le i\le m, 1\le k \le n\}.$$
Therefore, there are at least $|(\calI_n)^m|$ different  $(m,n)$-structure sets.
By \cite[Theorem 8]{chowla1951recursions},
the number of involutions in $\Sym(n)$ is $$ |\calI_n| \sim \exp(\tfrac{1}{2}n(\log n -1) + \sqrt{n}) \ge n^{\tfrac{1}{4} n}$$
for large $n$.
Thus the number of structure sets of degree $(m,n)$ is at least $$|(\calI_n)^m| \ge n^{\tfrac 14 mn} \ge (n^2)^{\tfrac 18 mn} \ge (mn)^{\tfrac 18 mn}$$
where the last inequality follows from $m\le n$. 
By choosing $\alpha$ small enough, we get the claim for all $m$ and $n$ (not just large enough $n$).
\end{proof}

\begin{theorem} \label{thm: count}
There exists a number $\alpha>0$ such that, for all sufficiently large natural numbers $m$ and $n$, there are at least  $(m n)^{\alpha m n}$ pairwise non-commensurable, involutive BMW groups $\Gamma$ of degree $(m,n)$ such that $\FR{\Gamma} = \Gamma^+$ is simple.
\end{theorem}

\begin{proof}
Note that the partial structure set $S_0$ has no element containing $\{a_i, b_k\}$ where $i\in \llbracket 11,m-3\rrbracket$ and $k\in \llbracket 12, n-3\rrbracket$.
Set $m',n'$ to be the number of elements in those integer intervals respectively -- namely $m'=m-13$ and $n'=n-14$. 
We see that we can add to $S_0$ any partial structure set supported on those elements.
Using Lemma \ref{lem: lower bound on labelled involutive BMW}, there is some $\alpha'>0$ so that there are at least $(m'n')^{\alpha' m'n'}$ different ways of extending $S_0$ to a partial structure set $S'$. By further adding the sets $\{a_i,b_j,a_i,b_j\}$ to $S'$ for any $i,j$ such that $\{a_i, b_j\}$ is not contained in an element of $S'$,
one obtains a structure set $S$. 
Therefore there are at least $(m'n')^{\alpha' m'n'}$ structure sets containing $S_0$. Additionally, a BMW group $\Gamma$ associated to such a structure set satisfies that $\FR{\Gamma} = \Gamma^+$ is simple by Theorem~\ref{thm: containing P implies simple of index 4}.

We now give a lower bound for the number of commensurability classes of involutive BMWs associated to structure sets containing $S_0$.
By Proposition~\ref{prop: upper bound on commensurable BMWs} the commensurability class of such a BMW group has at most $2(m!n!)^2$ structure sets. 
Therefore, by the previous paragraph, there are at least $(m'n')^{\alpha'm'n'} / 2(m!n!)^2$ commensurability classes of such involutive BMW groups of degree $(m,n)$.
By using $m! \le m^m$ and $n! \le n^n$ and $m'\ge m/2$ and $n'\ge n/2$ one gets the desired lower bound.
\end{proof}

\section{A Random Model for BMW groups}\label{sec:random_model}
\label{sec: random model}
For each even $n\in \bbN$, let $\FPF_{n}\subseteq \Sym(n)$ be the subset of \fpf involutions, i.e., involutions which do not fix any element.
When we write $\FPF_n$, it is implied that $n$ is even.
Let $(\FPF_n)^m$ be the set of $m$-tuples of \fpf involutions. 
We say that $\ualpha = (\alpha_1, \dots, \alpha_m) \in (\FPF_n)^m$ has \emph{triple matchings} if $\alpha_i(k) = \alpha_j(k) = \alpha_p(k)$ for some $1 \le k \le n$ and some distinct $1 \le i < j < p \le m$.

Fix $\ualpha = (\alpha_1, \dots, \alpha_m) \in (\FPF_n)^m$ with no triple matchings. 
We show how to 
canonically define an $(m,n)$-structure set $S_{\ualpha}$ with associated $B$-tree local involutions $\alpha_1, \dots, \alpha_m$.
After fixing a marking $\mathcal M$ of $T_m \times T_n$, by Proposition~\ref{prop:structure_set_bmw_bijection} this also defines (up to conjugacy) an involutive BMW group $\Gamma_{\ualpha} \in \BMW_{\mathcal M}(m,n)$ with structure set~$S_{\ualpha}$.

For each $1 \le k < l \le n$, set $I_{k,l} := \{i~|~\alpha_i(k) = l\}$. 
Note that since $\ualpha$ has no triple matchings, $|I_{k,l}|\le 2$. 
The structure set $S_{\ualpha}$ is the collection of subsets $\{a_i,b_k,a_j,b_l\}$ such that $1\le i \le j\le m$ and $1\le k< l\le n$ satisfy $I_{k,l}=\{i,j\}$ 
(note that $i$ could equal $j$).
It is straightforward to check that $S_{\ualpha}$
is indeed a structure set and that 
the $B$-tree local involutions of $S_{\ualpha}$ are exactly $\alpha_1,\ldots,\alpha_m$.

\begin{example} \label{ex: bmw_group}
    Suppose that $\ualpha = (\alpha_1, \alpha_2, \alpha_3) \in (\FPF_6)^3$ is such that $\alpha_1 = (12)(34)(56)$, $\alpha_2 = (12)(35)(46)$ and $\alpha_3 = (16)(35)(24)$. Note that $\ualpha$ has no triple matchings. The structure set associated to $\ualpha$ is then
    \begin{align*}
    S_{\ualpha} = \{ &\{a_1,b_1,a_2,b_2\}, \{a_1,b_3,a_1,b_4\}, \{a_1, b_5, a_1, b_6\}, \{a_2, b_3, a_3, b_5\}, \\
    &\{a_2, b_4, a_2, b_6\}, \{a_3, b_1, a_3, b_6\}, \{a_3, b_2, a_3, b_4\}.\}
    \end{align*}
    See Figure \ref{fig:bipartite_example}.
    \begin{figure}
        \centering
        \includegraphics{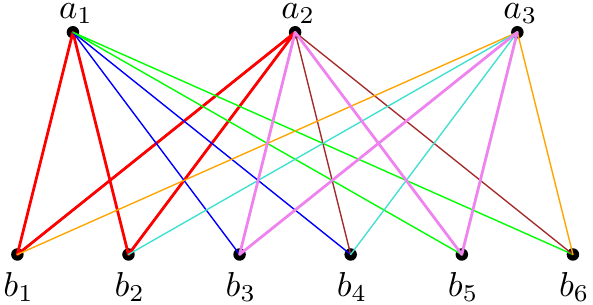}
        \caption{The partition of the edges of the bipartite graph $K_{3,6}$ corresponding to Example \ref{ex: bmw_group}}
        \label{fig:bipartite_example}
    \end{figure}
\end{example}

A \emph{random element of $\FPF_n$} is an element of $\FPF_n$ chosen uniformly at random. A \emph{random element of $(\FPF_n)^m$} is an element of $(\FPF_n)^m$ chosen uniformly at random, i.e., an $m$-tuple of $m$ independently chosen, random elements of $\FPF_n$.
We are now ready to define random involutive BMW groups:
\begin{definition}
    Suppose a marking $\mathcal M$ for $T_m \times T_n$ is fixed.
    Let $n>0$ be even and let $\ualpha$
    be a random element of $(\FPF_n)^m$.
    If $\ualpha$ has no triple matchings, we define the corresponding \emph{random involutive BMW group of degree $(m,n)$} to be $\Gamma_{\ualpha} \in \BMW_{\mathcal M}(m,n)$.
    On the other hand, if $\ualpha$ contains a triple matching, then we say that the corresponding random involutive BMW group is not defined.
\end{definition}

Let $\calP$ be a property of BMW groups. We say that a \emph{random involutive BMW group of degree $(m,n)$ satisfies property $\calP$
with probability $p$}, if given a 
random
$\ualpha = (\alpha_1, \dots, \alpha_m) \in (\FPF_n)^m$,
then with probability $p$ the corresponding random involutive BMW group is defined and satisfies property~$\calP$.

We will see in Lemma~\ref{lem:triple_matchings}, 
if $m$ is a function of $n$ satisfying $m(n) = o(n^{\tfrac13})$,
 then a
 random
 $m-$tuple $\ualpha  \in (\FPF_n)^{m(n)}$ 
 has no triple matchings (and consequently defines a random BMW group) with probability tending to $1$ as $n$ tend to infinity.

We first prove two elementary lemmas regarding fixed-point-free involutions.
Recall that the \emph{double factorial} of an integer $n$ is defined as $n!! := n\cdot (n-2)\cdot (n-4) \dots 2$ for $n$ even and as $n!! := n\cdot(n-2)\cdot(n-4)  \dots  3\cdot 1$ for $n$ odd.

\begin{lemma}\label{lem:count_fpf}
	For $n$ even, $\lvert \FPF_{n}\rvert=(n-1)!!$. 
\end{lemma}
\begin{proof}
	Let $\phi(n) = |\FPF_n|$, and let $\sigma \in \FPF_n$.
	There are $n-1$ options for $\sigma(1)$. After choosing $\sigma(1)$, there are $\phi(n-2)$ ways of completing $\sigma$ to a fixed-point-free involution. So we get the recursion formula $\phi(n)=(n-1)\phi(n-2)$, with $\phi(2)=1$. This gives that $\phi(n)=(n-1)!!$.
\end{proof}

\begin{lemma} \label{lem:containing_transpositions}
	For $n$ even, let $O_1,\dots, O_k$ be a collection of unordered pairs of distinct elements 
	in $\{1,\dots, n\}$ such that $O_i \cap O_j = \emptyset$ for $i \neq j$.  The probability that 
	a  
	random element
	of $\FPF_{n}$ contains the orbit $O_i$ for all $1 \le i \le k$
    is $\frac{(n-2k - 1)!!}{(n-1)!!}$.
\end{lemma}
\begin{proof}
	By Lemma \ref{lem:count_fpf}, there are $(n-1)!!$ fixed-point-free involutions in $S_{n}$,  $(n-2k-1)!!$ of which have $O_i$ as an orbit for every $i$. 
\end{proof}  

The next lemma shows that when $n$ is sufficiently greater than $m$,  there are no triple matchings with high probability.
\begin{lemma}\label{lem:triple_matchings}
	 A 
	 random
	 $\ualpha\in(\FPF_{n})^{m}$
	 has no triple matchings with probability at least $1-\frac{4m^3}{n}$.
\end{lemma}
\begin{proof}
	Let $\ualpha=(\alpha_1,\dots,\alpha_m)$. Let $p$  denote the probability that  $\ualpha$ has a triple matching. 
     Let $\Omega\subseteq [\![1,m]\!]^3$ be the set of triples $(i,j,k)$ with $i<j<k$. For each $\omega=(i,j,k)\in \Omega$, let $Z_{\omega,l}$ denote the event where $\alpha_i(l)=\alpha_j(l)=\alpha_k(l)$.
     Each such event has probability  $\frac{1}{(n-1)^2}$ of occurring.
	Let $Z=\bigcup_{\omega\in \Omega}\bigcup_{l=1}^{n}Z_{\omega,l}$, and note that the probability of $Z$ occurring is $p$.
	As $\lvert\Omega\rvert\leq m^3$, we deduce via a union bound that $p\leq\frac{m^3n}{(n-1)^2}\leq \frac{4m^3}{n}$ where we used that $\frac{1}{n-1}\leq \frac{2}{n}$ for $n \ge 2$.
\end{proof}

\section{$A$-tree local actions}\label{sec:atree_local}
The aim of this section is to prove the following:
\begin{theorem}\label{thm: A_tree_generation_alternative}
	There is a constant $C$ such the following holds. If $\Gamma$ is a random BMW involution group of degree $(m,n)$ with $n>m^5$, then the $A$-tree local action of $\Gamma$ is $\Sym(m)$   with probability at least $1 - \frac{C}{m}$.
\end{theorem}

In the next proposition, we give conditions on  $\ualpha\in(\FPF_{n})^{m}$ that   ensure  the $A$-tree local action of the  induced BMW group $\Gamma_{\ualpha}$ is $\Sym(m)$. We then show that  all these conditions hold with sufficiently high probability.

We say that  $\ualpha=(\alpha_1,\dots, \alpha_m)\in(\FPF_{n})^{m}$ has \emph{overlapping matches} if there exist distinct pairs $\{i,j\},\{i',j'\}$, with $i\neq j$ and $i'\neq j'$, such that $\alpha_i(k)=\alpha_{j}(k)$ and  $\alpha_{i'}(k)=\alpha_{j'}(k)$ for some $k$.

\begin{proposition}\label{prop:A-side_properties}
Suppose that $\ualpha=(\alpha_1,\dots, \alpha_m)\in(\FPF_{n})^{m}$ satisfies:
	\begin{enumerate}[({A}1)]
		\item \label{item:no_triple_matchings} %(no triple matching) 
		$\ualpha$ has no triple matchings,
		\item \label{item:no_overlapping_matches} $\ualpha$ has no overlapping matches, and
		\item \label{item:lots_of_common_orbits}for all  $i,i'$, there exists $j$ such that $\alpha_i$ and $\alpha_j$ share a common orbit, and $\alpha_j$ and $\alpha_{i'}$ share a common orbit.
	\end{enumerate}
Then the $A$-tree local action of $\Gamma_{\ualpha}$ is $\Sym(m)$.
\end{proposition}
\begin{proof}
    As $\ualpha$ has no triple matchings, we let $\Gamma=\Gamma_{\ualpha}$ and $S= S_{\ualpha}$
    be respectively the associated BMW group and structure set associated to $\ualpha$. 
    Let $\beta_1,\dots, \beta_n$
    be the $A$-tree local involutions of $\Gamma$.
    
	Suppose that $\alpha_i(k)=\alpha_j(k)=l$ for some $1\leq i < j\leq m$ and $1\leq k < l\leq n$.
	We claim that $\beta_k$ is the transposition $(i\;j)$. 
	By the definition of $S$, $\{a_i,b_k,a_j,b_l\} \in S$, and it follows from this that $\beta_k(i)=j$.  
	By
	\ref{item:no_overlapping_matches},
	for all distinct $i', j' \neq i,j$, we have that $\beta_{i'}(k) \neq \beta_{j'}(k)$.
	From this it follows that for all $i' \neq i,j$, $\{a_{i'},b_k,a_{i'},b_{\alpha_{i'}(k)}\} \in S$.
	Thus, $\beta(i') = i'$ for all $i' \neq i,j$. We conclude that $\beta_k$ is indeed the transposition $(i\;j)$, showing the claim.

     Let $1\leq i < i'\leq m$. By \ref{item:lots_of_common_orbits}, there exists some $j\neq i,i'$ such that $\alpha_i$ and $\alpha_j$ share a common orbit, and $\alpha_j$ and $\alpha_{i'}$ share a common orbit. By the previous paragraph, this implies there exist $1\leq l,l'\leq n$ such that $\beta_l$ is the transposition $(i\;j)$ and $\beta_{l'}$ is the transposition $(j\;i')$. 
     Thus $\beta_{l} \beta_{l'} \beta_{l}$ is the transposition $(i\;i')$.
     Consequently, the subgroup generated by $\{\beta_1,\dots, \beta_n\}$  contains all transpositions and so generates $\Sym(m)$.
\end{proof}

The remainder of this section is devoted to showing that a 
random
$\ualpha\in\FPF_{n}^{m}$ satisfies the three conditions from Proposition \ref{prop:A-side_properties}  with sufficiently high probability.  
Condition \ref{item:no_triple_matchings} was shown to hold in Lemma \ref{lem:triple_matchings}. We thus begin with condition \ref{item:no_overlapping_matches}.

\begin{lemma}\label{lem:overlapping_matches}
	 A 
	 random
	 $\ualpha\in(\FPF_{n})^{m}$ 
	 has no overlapping matches with probability at least $1-\frac{4m^4}{n}$.
\end{lemma}
\begin{proof}
	Let $\ualpha=(\alpha_1,\dots,\alpha_m)$. Let  $p$ denote the probability that $\ualpha$ has overlapping matches. 
    Let $\Pi\subseteq [\![1,m]\!]^4$ be the set of quadruples $(i,i',j,j')$ such that $\{i,i'\}\neq \{j,j'\}$, $i\neq i'$ and $j\neq j'$.  For each $\pi=(i,i',j,j')\in \Pi$, let $Y_{\pi,k}$ be the event where $\alpha_i(k)=\alpha_{i'}(k)$ and  $\alpha_{j}(k)=\alpha_{j'}(k)$. The probability that each such event occurs is equal to  $\frac{1}{(n-1)^2}$. 
	Let $Y=\bigcup_{\pi\in \Pi}\bigcup_{k=1}^{n}Y_{\pi,k}$, and note that the probability that $Y$ occurs is equal to $p$. Since $\lvert \Pi\rvert\leq m^4$, we deduce via  union bound that  $p \le \frac{m^4n }{(n-1)^2}\leq \frac{4m^4}{n}$.
\end{proof}

Next, we give the probability that two fixed-point-free involutions share a common orbit.

\begin{lemma}\label{lem:prob_match}
	The probability that two random elements of $\FPF_{n}$ share
	a common orbit is
	\[\sum_{k=1}^{\frac{n}{2}} (-1)^{k+1} {\frac{n}{2} \choose k} \frac{(n - 2k - 1)!!}{(n-1)!!},\] 
	Moreover, this probability converges to $1-e^{-\frac{1}{2}}$ as $n\rightarrow \infty$.
\end{lemma}
\begin{proof}
	Suppose $\alpha,\alpha'\in \FPF_{n}$ are chosen uniformly at random.
	Let $\{O_1,\dots, O_r\}$ be the set of orbits of $\alpha$ where $r = \frac{n}{2}$. Let $T_i\subseteq \FPF_n$ be the set of \fpf involutions with orbit $O_i$, and let $Z$ be the number of \fpf involutions in $\FPF_{n}$
	with orbit  $O_i$ for some $1\leq i\leq n$. By inclusion-exclusion:
	\[Z = \left|\bigcup_{i=1}^{r} T_i \right| = \sum_{k=1}^{r} (-1)^{k+1} \left( \sum_{1 \le n_1 < \dots < n_k \le r} |T_{n_1} \cap \dots \cap T_{n_k}|\right)\]
	
	By Lemma~\ref{lem:containing_transpositions}, $|T_{n_1} \cap \dots \cap T_{n_k}| = (n - 2k - 1)!!$ for every $1 \le n_1 < \dots < n_k \le r$. By  the above equation, we then have:
	\[Z =  \sum_{k=1}^{r} (-1)^{k+1} {r \choose k} (n - 2k - 1)!!\]
	By Lemma \ref{lem:count_fpf}, we can divide by $(n  - 1)!!$ to conclude the first claim.
	
	We now prove the convergence claim.
	Set $a_{k,r}=(-1)^{k+1} {r \choose k} \frac{(2r - 2k - 1)!!}{(2r-1)!!}$ for $k\leq r$ and $a_{k,r}=0$ otherwise. We want to show that  $\lim_{r\rightarrow \infty} \sum_{k=1}^ra_{k,r}$ is equal to $1-e^{-\frac{1}{2}}$.
	
	We first note that  for all $k\leq r$, \begin{align*}
		a_{k,r}&=(-1)^{k+1} {r \choose k} \frac{1}{(2r-1)(2r-3)\dots (2r-2k+1)}\\
			&=\frac{(-1)^{k+1}}{k!}\frac{r(r-1)\dots (r-k+1)}{(2r-1)(2r-3)\dots (2r-2k+1)}\\
			&=\frac{(-1)^{k+1}}{k!}\prod_{i=0}^{k-1}\frac{r-i}{2r-2i-1}
	\end{align*}
	Thus $\lvert a_{k,r}\rvert\leq \frac{1}{k!}$ for all $k$ and $r$. Since $\sum_{k=1}^\infty\frac{1}{k!}<\infty$, Tannery's theorem implies that $\lim_{r\rightarrow \infty} \sum_{k=1}^\infty a_{k,r}$ exists and is equal to $ \sum_{k=1}^\infty(\lim_{r\rightarrow \infty}a_{k,r})$. Clearly $\lim_{r\rightarrow \infty}a_{k,r}=\frac{(-1)^{k+1}}{k!2^k}$.
	Therefore \begin{align*}
		\lim_{r\rightarrow \infty} \sum_{k=1}^ra_{k,r}&=\sum_{k=1}^\infty \frac{(-1)^{k+1}}{k!2^k} = 1-\sum_{k=0}^\infty \frac{1}{k!}(\frac{-1}{2})^{k}=1-e^{-\frac{1}{2}}.\qedhere
	\end{align*} 
\end{proof}

\begin{corollary} \label{cor:one_third}
    There exists a number $N$ such that whenever $n \ge N$, the probability that two random elements in $\FPF_n$ share a common orbit is at least $\frac{1}{3}$.
\end{corollary}
\begin{remark}
    It can be shown using estimates that $N$ in the previous corollary, can be taken to be $2$.
\end{remark}

Finally, we show that the third property of Proposition~\ref{prop:A-side_properties} holds with high probability.

\begin{lemma}\label{lem:many_matches}
	Let $N$ be as in Corollary~\ref{cor:one_third}, $n \ge N$ and   $\ualpha\in(\FPF_{n})^{m}$ be a   random element.
	 Then with probability at least
	 $1 -2m^2\left(\frac{8}{9}\right)^{m}$ the following property holds: for all $1 \le i < i' \le m$ there exists a $j$ such that $\alpha_i$ and $\alpha_j$ share a common orbit, and $\alpha_j$ and $\alpha_{i'}$ share a common orbit.
\end{lemma}

\begin{proof}
    Let
    $\ualpha = (\alpha_1, \dots, \alpha_m)$.
	For each $1 \le i < i' \le m$ and $j\neq i,i'$, let $Y_{i,j,i'}$ be the event that both  $\alpha_i$ and $\alpha_j$  share a common orbit, and $\alpha_j$ and $\alpha_{i'}$ share a common orbit. 
	Since (up to conjugation) we can treat $\alpha_{j}$ as fixed, and $\alpha_i$ and $\alpha_{i'}$ as independently randomly chosen, we get that the event that $\alpha_i$ and $\alpha_j$ share an orbit and the event that $\alpha_j$ and $\alpha_{i'}$ share an orbit are independent.
	By Corollary~\ref{cor:one_third} each of them occurs with probability at least $\frac{1}{3}$. 
	Therefore, the probability of the event $Y_{i,j,i'}$ is $\bbP(Y_{i,j,i'}) \ge \left(\frac{1}{3}\right)^2 = \frac{1}{9}$.

	Let $Z_{i,i'}$ be the event that no $Y_{i,j,i'}$ occurs for any $j\notin \{i,i'\}$, and let $Z= \bigcup_{i \neq i'} Z_{i, i'}$.	
	Observe that the property in the lemma's statement holds exactly when the event $Z$ does not occur.
    Since $Z_{i,i'} = \bigcap _{j\notin \{i,i'\}} Y_{i,j,i'}^c$ 
    and the events $Y_{i,j,i'}$ are independent, we have that
	\[\bbP(Z_{i,i'}) \le \left(1-\frac{1}{9}\right)^{m-2} =\left(\frac{8}{9}\right)^{m-2}\leq 2 \left(\frac{8}{9}\right)^{m},\]
    and by a union bound that
	\[ \bbP(Z) \leq \sum_{i \neq i'} \bbP(Z_{i,i'}) \leq  2m^2\left(\frac{8}{9}\right)^{m}. \qedhere\]
\end{proof}

\begin{proof}[Proof of Theorem \ref{thm: A_tree_generation_alternative}]
    Suppose first that $n \ge N$ where $N$ is as in Corollary~\ref{cor:one_third}. 
	By Lemmas \ref{lem:triple_matchings}, \ref{lem:overlapping_matches} and \ref{lem:many_matches}, and by a union bound, a 
	random
	$\ualpha\in \FPF_{n}^{m}$ satisfies the three properties of Proposition \ref{prop:A-side_properties} with probability at least 
	\[ p(m,n)\coloneqq 1 - \frac{4m^3}{n} - \frac{4m^4}{n} - 2m^2\left(\frac{8}{9}\right)^{m}.\] 
	Since $n> m^5$, we can pick some constant $C$ such that $p(m,n)\geq 1-\frac{C}{m}$ as required. Moreover, by choosing $C$ large enough, we can guarantee that this holds for all $n$ (not just for $n \ge N$).
\end{proof}

\section{$B$-tree local action}\label{sec:btree_local}

In this section we show that the $B$-tree local action of a random BMW group contains the alternating group $\Alt(n)$ asymptotically almost surely (Corollary \ref{cor: full B local action}).
This follows from a generation result for random fixed-point-free involutions (Theorem~\ref{thm: A_side_generation}).
This theorem should be compared to those of Dixon \cite{dixon1969probability} and  Liebeck-Shalev \cite{liebeck1996classical} which address generation results for random permutations and random involutions respectively.
In fact, our proof closely follows that of Liebeck-Shalev.

\begin{theorem}
\label{thm: A_side_generation}
Let $m\ge 3$, $\epsilon >0$ and $\ualpha = (\alpha_1,\ldots,\alpha_m)\in (\FPF_n)^m$ 
a random element.
Then 
$$\bbP (\Alt(n) \le \gen{\alpha_1,\ldots,\alpha_m}) \ge 1-O(n^{-(m\cdot(1/2-\epsilon)-1)}).$$
\end{theorem}

\begin{corollary}\label{cor: full B local action}
    Let $m = m(n)$ be a function of $n$ satisfying
    $5\le m(n) = O(n^{1/3})$, and let $\ualpha \in (\FPF_n)^m$ be a random element. 
   The probability that the $B$-tree local action of the random BMW group $\Gamma_{\ualpha}$ contains $\Alt(n)$ is 
   $1- O(1/n)$. 
\end{corollary}

\begin{proof}
By Lemma \ref{lem:triple_matchings}, the group $\Gamma_{\ualpha}$ is well-defined with sufficiently high probability. As we saw in \S\ref{sec: random model}, the $B$-tree local action is generated by $\alpha_1,\ldots,\alpha_m$. The bound above now follows from the theorem.
\end{proof}

To prove Theorem \ref{thm: A_side_generation}, we follow the same strategy as that of Liebeck-Shalev~\cite{liebeck1996classical}.
Given a group $G$, we let  $\calM_G$ denote the set of maximal, proper subgroups of $G$. 
Let $\calN_n \subset \calM_{\Sym(n)}$ be the set of maximal proper subgroups of $\Sym(n)$ which do not contain $\Alt(n)$.
If $\Alt(n) \not\le \gen{\alpha_1,\ldots,\alpha_m}$, then the permutations $\alpha_1,\ldots,\alpha_m$ are contained in some subgroup $M \in \calN_n$.
It follows by a union bound that 
\begin{align}
\bbP (\Alt(n) \not\le \gen{\alpha_1,\ldots,\alpha_m}) \le \sum_{M\in\calN_n} \bbP(\alpha_1,\ldots,\alpha_m \in M ) = \sum_{M\in\calN_n} \bbP(\alpha \in M)^m
\label{eq_does_not_contain_An}
\end{align}
where $\alpha \in \FPF_{n}$ is chosen uniformly at random.

Given a group $G$, Liebeck-Shalev define the function
$$ \zeta_G(s) := \sum_{M\in \calM_G} [G:M]^{-s},$$
and show that $\zeta_{\Alt(n)}(s) = O(n^{-(s-1)}) \to 0$ as $n\to \infty$ for all $s>1$ \cite[Theorem~3.1]{liebeck1996classical}.
A similar proof shows that 
\[\sum _ {M\in \calN_n} [\Sym(n):M]^{-s} = O(n^{-(s-1)})\] as $n\to \infty$ for all $s>1$.
By Lemma~\ref{lem:ls_lemma} below and  \eqref{eq_does_not_contain_An} we have that: 
\[\bbP (\Alt(n) \not\le \gen{\alpha_1,\ldots,\alpha_m}) \le \sum_{M\in\calN_n} \bbP(\alpha \in M)^m \le c^m \sum_{M\in\calN_n} [\Sym(n):M]^{-m(\frac{1}{2} - \epsilon)} \]
for some constant $c$.
Theorem~\ref{thm: A_side_generation} then follows from the last two equations.
Thus, we now turn our attention to proving the following:

\begin{lemma} \label{lem:ls_lemma}
For every $\epsilon>0$ there exists a constant $c$ such that given any even integer $n>0$, any $M\in\calN_{n}$ and a random $\alpha \in \FPF_{n}$, then
$$\bbP(\alpha \in M) \le c[\Sym(n):M]^{-\tfrac{1}{2}+\epsilon}$$
\end{lemma}

\begin{proof}
We follow the same outline as the proof of \cite[Theorem 5.1]{liebeck1996classical}.
We have the following two equations
\begin{align*}
    |\Sym(n)| &=n! = \exp(n\log n - n + \tfrac{1}{2}\log (2\pi n)+o(1)) \\
    |\FPF_{n}| &= (n-1)!! = \exp (\tfrac{1}{2}(n\log n - n) + O(1))
\end{align*}
each following from Stirling's approximation, where for the second equation we also use the identity
$$|\FPF_{n}|=(n-1)!! = \frac{(n)!}{2^{n/2} (n/2)!}.$$

Thus, there exists a constant $c_0\ge 1$ such that 
\begin{equation}\label{eq: tighter bound on the number of fpf}
    c_0\ii|\Sym(n)|^{\tfrac{1}{2}} \le  |\FPF_{n}|\cdot (2\pi n)^{\tfrac{1}{4}}\le c_0 |\Sym(n)|^{\tfrac{1}{2}}
\end{equation}
Additionally, since $(2\pi n)^\frac{1}{4} \ll n! = |\Sym(n)|$, we may also assume that $c_0$ satisfies
\begin{equation}\label{eq: bound on the number of fpf}
    c_0\ii|\Sym(n)|^{\tfrac{1}{2}-\epsilon/2} \le  |\FPF_{n}|\le c_0|\Sym(n)|^{\tfrac{1}{2}}
\end{equation}

We now consider three cases depending  on the structure of the subgroup $M$.
\\\\
\textbf{Case 1:} $M$ is primitive.

It is shown in \cite{praeger1980orders} that every primitive subgroup $M\in \calN_n$ satisfies $|M|\le 4^n$. Furthermore, there exists a constant $c_1$ such that $4^n \le c_1|\Sym(n)|^{\epsilon/2}$.
By \eqref{eq: bound on the number of fpf}:
$$ \frac{|M\cap \FPF_{n}|}{|\FPF_{n}|} \le \frac{|M|}{|\FPF_{n}|} \le c_0 c_1 \frac{|\Sym(n)|^{\epsilon/2}}{|\Sym(n)|^{\tfrac{1}{2} - \epsilon/2}} \le c_0 c_1 |\Sym(n)|^{-\frac{1}{2}+\epsilon} \le c_0c_1 [\Sym(n):M]^{-\frac{1}{2} + \epsilon}$$
\\\\
\textbf{Case 2:} $M$ is not transitive. 

In this case, $M$ can be identified with $\Sym(k) \times \Sym(l)$ for some $k,l< n$ such that $k+l=n$.
Furthermore, $M\cap \FPF_{n} = \FPF_{k} \times \FPF_{l}$ if both $k$ and $l$ are even, and $M\cap \FPF_{n} = 1$ otherwise.  By \eqref{eq: tighter bound on the number of fpf} we get that
\begin{align*}
 \frac{|M\cap \FPF_{n}|}{|\FPF_{n}|} &\le \frac{|\FPF_{k}||\FPF_{l}|}{|\FPF_{n}|}
 \le c_0^3\left(\frac{|\Sym(k)||\Sym(l)|}{|\Sym(n)|} \right)^{\tfrac{1}{2}} \cdot \left(\frac{ n}{2 \pi kl} \right)^{\tfrac{1}{4}}\\
 &\le c_0^3 \left( \frac{|M|}{|\Sym(n)|}\right)^{\tfrac{1}{2}} \le c_0^3 [\Sym(n):M]^{-\tfrac{1}{2}}
\end{align*}
where the third inequality follows since $n\le 2\pi k l$. 
\\\\
\textbf{Case 3:} $M$ is transitive and imprimitive. 

In this case, $M$ can be identified with $\Sym(k) \wreath \Sym(l)$ for some $k,l< n$ such that $kl=n$.
In wreath product notation, every permutation $\alpha\in M$ can be written as $\alpha=(\pi_1,\ldots,\pi_l)\cdot\tau$ where $\pi_i \in \Sym(k)$ and $\tau\in \Sym(l)$.
It readily follows that $\alpha$ is a \fpf involution if and only if the following three conditions hold:
\begin{enumerate}
    \item Up to relabeling the $l$--element set that $\Sym(l)$ acts on, $\tau$ has the form
    \[\tau = (1 \; 2) (3\; 4)\ldots (2m-1 \; 2m)\] 
    for some $m \le \frac{l}{2}$,
    \item $\pi_1=\pi_2 \ii, \ldots, \pi_{2m-1} = \pi_{2m}\ii$, and
    \item  $\pi_{2m+1},\ldots,\pi_l$ are \fpf involutions in $\Sym(k)$.
\end{enumerate}

Fixing an involution $\tau \in \Sym(l)$ with $m$ transpositions,  the number of \fpf involutions $\alpha\in M \cap \FPF_{n}$ which can be written as $\alpha = (\pi_1,\ldots,\pi_l)\cdot\tau$ is 
\begin{equation*}
    |\Sym(k)|^{m}|\FPF_{k}|^{l-2m} = (k!)^m((k-1)!!)^{l-2m} \le c_0^{l-2m}(k!)^{l/2} \le c_0^{l}(k!)^{l/2}.
\end{equation*}
where the first inequality follows from \eqref{eq: bound on the number of fpf}. We thus get the bound:
$$ |M\cap \FPF_{n}|\le c_0^l (k!)^{l/2}\cdot |\calI_l|$$
where $\calI_l$ is the set of all involutions in $\Sym(l)$.
It is shown in  \cite{moser1955solutions} that 
$$|\calI_l| \le \exp({\tfrac{1}{2}l\log l -\tfrac{1}{2}l + \sqrt l +O(1)})\le  \exp({\tfrac{1}{2}l\log l -\tfrac{1}{2}l + \tfrac{1}{2}\sqrt {2\pi l} +O(1)}).$$
Therefore, by Stirling's approximation, there exists $c_2$ such that $|\calI_l| \le c_2(l!)^{\tfrac{1}{2}}$
for all $l$.
From the last two equations we deduce that:
$$|M\cap \FPF_{n}| \le c_2  c_0^l (k!)^{l/2} (l!)^{\tfrac{1}{2}}= c_2 c_0^l ((k!)^l l!)^{\tfrac{1}{2}}.$$
By \eqref{eq: tighter bound on the number of fpf} and as $|M|=(k!)^l l!$,  we get 
\begin{equation}\label{eq: first bound in case 2}
\frac{|M\cap \FPF_{n}|}{|\FPF_{n}|} \le  c_0 c_2 c_0^l (2\pi n)^{\tfrac 14}\left(\frac{(k!)^l l! }{n!} \right)^{\tfrac{1}{2}} \le c_2 c_0^{l+1} (2\pi n)^{\tfrac 14} [\Sym(n):M]^{-\frac{1}{2}}.
\end{equation}

Recall that $M = \Sym(k) \wreath \Sym(l)$ is a maximal subgroup of $\Sym(n)$ that preserves a partition of $n$ elements into $l$ subsets of size $k$.
Without loss of generality, let the $n$-element set be $\{1,\ldots,k\}\times \{1,\ldots,l\}$ and suppose that $M$ preserves the partition $\bigsqcup_{j=1}^l \{1,\ldots,k\}\times \{j\}$.

Consider the subgroup $H$ of $\Sym(n)$ that stabilizes the set $\{i\} \times \{1,\ldots,l\}$, for each $1 \le i \le k$ and which fixes point-wise the set $\{k\} \times \{1,\ldots,l\}$. Then $H$ is a copy of $(\Sym(l))^{k-1}$ which satisfies $H\cap M=1$. Thus
\begin{equation}\label{eq: good index bound}
    [\Sym(n):M]\ge |H|= (l!)^{k-1}
\end{equation}
Since $l!$ is super-exponential, there exists $c_3$ so that $c_0^l \le c_3 \cdot (l!)^{\epsilon/2}$. We get that
\begin{equation}\label{eq: bound on c_0^l}
    c_0^l \le c_3 \cdot (l!)^{\epsilon/2} \le c_3 [\Sym(n):M]^{\epsilon/2}
\end{equation}
where the second inequality follows as $[\Sym(n):M] \ge (l!)^{k-1} \ge l!$. 
As $l!\ge \frac{2^l}{2}$, as $n=kl$ and as $l \le \frac{n}{2}$, we also get that
$$[\Sym(n):M]\ge (l!)^{k-1} \ge \frac{2^{l(k-1)}}{2} \ge \frac{2^{n/2}}{2}$$
Moreover, there exists a constant $c_4$ such that $(2\pi n)^{\tfrac 14} \le \frac{c_4}{2} (2^{n/2})^{\epsilon/2}$. We then get that
\begin{equation}\label{eq: bound on n^-1/4}
    (2\pi n)^{\tfrac 14} \le \frac{c_4}{2} (2^{n/2})^{\epsilon/2} \le c_4 [\Sym(n):M]^{\epsilon/2}
\end{equation}

By \eqref{eq: first bound in case 2}, \eqref{eq: bound on c_0^l} and \eqref{eq: bound on n^-1/4} it follows that
\begin{equation*}
\frac{|M\cap \FPF_{n}|}{|\FPF_{n}|} \le c_0 c_2 c_3 c_4 [\Sym(n):M]^{-\tfrac 12 + \epsilon} 
\end{equation*}

Setting $c=\max\{c_0 c_1, c_0^3,c_0 c_2 c_3 c_4\}$ completes the proof.
\end{proof}

\section{Irreducibility of random BMWs}\label{sec:irreducibility}
The aim of this section is to complete the proof of Theorem \ref{main thm: random model}. More precisely, we prove the following:
\begin{theorem}\label{thm:irred}
	There is a constant $C$ such the following holds. If $\Gamma$ is a random BMW involution group of degree $(m,n)$ with $n>m^5$, then all of the following hold with probability at least $1-\frac{C}{m}$:
	\begin{enumerate}
		\item \label{it:full_a_local} the $A$-tree local action of $\Gamma$ is $\Sym(m)$;
		\item \label{it:full_b_local} the $B$-tree local action of $\Gamma$ is either $\Sym(n)$ or $\Alt(n)$;
		\item \label{it:irreducible} $\Gamma$ is irreducible;
		\item \label{it:just_infinite} $\Gamma$ is hereditarily just-infinite.
	\end{enumerate}
\end{theorem}
Conclusions (\ref{it:full_a_local}) and (\ref{it:full_b_local}) follow directly from Theorem \ref{thm: A_tree_generation_alternative} and Corollary~\ref{cor: full B local action} respectively. Moreover, conclusion (\ref{it:just_infinite}) follows from conclusions (\ref{it:full_a_local})--(\ref{it:irreducible}) and Theorem~\ref{thm: BM just infinite}. Thus, we are left to prove conclusion (\ref{it:irreducible}) regarding the irreducibility of $\Gamma$.
By a theorem of Caprace \cite[Theorem~1.2(vi)]{caprace2018radius}, conclusion (3) is implied by conclusions (1) and (2) as long as
\begin{align*}
    n \notin \{&\frac{m!}{2} - 1, \;\frac{m!}{2},\; m! - 1,\; \frac{m!(m-1)!}{4}-1,\; \frac{m!(m-1)!}{4},\\ &\frac{m!(m-1)!}{2} - 1,\; \frac{m!(m-1)!}{2},\; m!(m-1)!-1\}.
\end{align*}
Thus, in order to finish the proof Theorem~\ref{thm:irred}, it is enough to show that (3) holds whenever $n$ is one of the values above. To do so, we actually show conclusion (3) holds whenever $n>m^8$ (covering the above finite cases) by using the theorem of Trofimov--Weiss stated below. 

Suppose that $\Gamma$ is a group acting vertex-transitively on a locally finite connected  graph $X$. 
We do not assume the action is faithful. Given a vertex $v\in X$, let $\Gamma_v$ be its stabilizer, and let $\Gamma^{[i]}_v$ be the pointwise stabilizer of 
the set of 
all vertices distance
$i$ or less
from $v$. 
Recall that the \emph{local action} of $\Gamma$  is the subgroup of $\Sym(m)$ induced by the action of $\Gamma_v$ on the edges adjacent to $v$. The following is a consequence of a theorem of Trofimov--Weiss, as reformulated by Caprace \cite[\S 4.5]{Caprace-survey}.
\begin{theorem}[{\cite[Theorem 1.4]{trofimovweiss1995graphs}}]\label{thm:trofweiss}
		Suppose that a group 
		$\Gamma$ acts vertex transitively on a connected locally finite graph $X$ with  2-transitive local action. If $\Gamma_v^{[6]}\not\leq \Gamma_v^{[7]}$ for some $v\in V(X)$, then   the image of the action $\Gamma\rightarrow \Aut(X)$ is not discrete.	
\end{theorem}

Recall that a BMW group $\Gamma\leq \Aut(T_m)\times \Aut(T_n)$ acts on $T_m$ by  projecting to the first factor.
We prove Theorem \ref{thm:irred}~(\ref{it:irreducible})
by showing that the hypotheses of  Theorem \ref{thm:trofweiss} are satisfied with sufficiently high probability for a random BMW group. We do this by investigating the following graph:

\begin{definition}
	Given $\ualpha\in (\FPF_n)^m$, define the following simplicial 
	graph  $\calG_{\ualpha}$ whose edges are colored black and white such that:
	\begin{itemize}
		\item The vertex set of $\calG_{\ualpha}$ is $B_n=\{b_1,\dots, b_n\}$.
		\item Vertices $b_i$ and $b_j$ are joined by an edge if there is some $1\leq k\leq m$ such that $\alpha_k(i)=j$. This edge is black if there exist distinct $k\neq k'$ such that $\alpha_k(i)=\alpha_{k'}(i)=j$ and is white otherwise.
	\end{itemize}
\end{definition}
An example of the above graph is shown in Figure \ref{fig: schreier}

\begin{figure}
    \centering
    \includegraphics{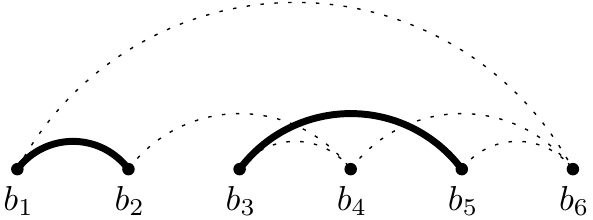}
    \caption{The graph $\calG_{\ualpha}$ for the permutations in Example~\ref{ex: bmw_group}. Bold edges represent `black' edges and dotted edges represent `white' edges.}
    \label{fig: schreier}
\end{figure}

\begin{lemma}\label{lem:irred_conditions}
	Suppose that $\ualpha=(\alpha_1,\dots, \alpha_m)\in  (\FPF_n)^m$ satisfies:
	\begin{enumerate}[({Irr}1)]
		\item \label{item:irr1} $\ualpha$ has no triple matchings;
		\item \label{item:irr2} there exists some $b\in B_n$ such that  all edges in the closed ball $N_6(b)$ of $\calG_{\ualpha}$ are white; 
		\item \label{item:irr2a} $\calG_{\ualpha}$ is connected; 
		\item \label{item:irr2b} $\calG_{\ualpha}$ contains a black edge; 
		\item \label{item:irr3} the $A$-tree local action of $\Gamma_{\ualpha}$ is 2-transitive.
	\end{enumerate}
Then the BMW group $\Gamma_{\ualpha}$ is irreducible.
\end{lemma}
\begin{proof}
	Recall from Section \ref{sec:structure_sets} that the action of the involutive BMW group $\Gamma=\Gamma_{\ualpha}$ on $T_m\times T_n$ preserves
	a labeling of edges of $T_m\times T_n$, where horizontal edges are labeled by elements of $A_m$ and vertical edges by $B_n$. 
	Moreover, recall that the 1-skeleton of $T_m\times T_n$ can be identified with the Cayley graph of $\Gamma$. 
	Let $o=(o_A,o_B)\in T_m\times T_n$ be the vertex corresponding to the identity.  
	By the definition of $\Gamma_{\ualpha}$ and  as $\ualpha$ has no triple matching, if $a_i,b_k,a_j,b_l$  are the labels of the edges of a square in $T_m\times T_n$, then $\alpha_i(k)=\alpha_j(k)=l$. 
	Let $\pi_A:T_m\times T_n\to T_m$ be the projection map. 

	By a slight abuse of notation, we identify  each $b\in B_n$ with the element $\Gamma$  that interchanges
	the endpoints of the edge  incident to $o$
	and  labeled by $b$. We  determine the action of $b$ on $T_m$ as follows. Let $L$ be a path in $T_m\times\{o_B\}$ starting at $o$.
	Let $a_{i_1},\dots, a_{i_r}\in A_m$ be the labels of consecutive edges of $L$. 
	Now let $R$ be the unique $1\times r$ rectangle in $T_m\times T_n$  whose bottom left vertex is $o$, whose left edge is labelled by $b$ and whose top edge is the path $bL$.  Such a rectangle is shown in Figure \ref{fig:rectangle}, with $a_{i'_j}$ and $b_{k_j}$ as indicated in Figure~\ref{fig:rectangle}. 
	
	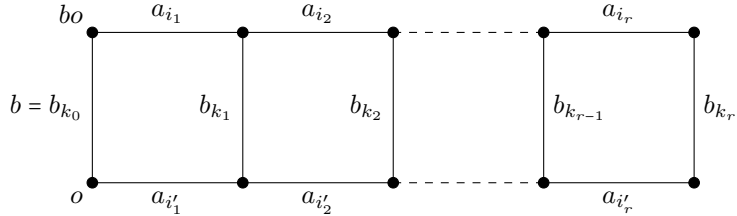
\begin{figure}[h]
		\begin{center}
	\begin{tikzpicture}
\node  at (1,2)  [above] {$a_{i_1}$};
\node  at (3,2)  [above] {$a_{i_2}$};
\node  at (7,2)  [above] {$a_{i_{r}}$};
\node  at (1,0)  [below] {$a_{i'_1}$};
\node  at (3,0)  [below] {$a_{i'_2}$};
\node  at (7,0)  [below] {$a_{i'_{r}}$};

\node  at (0,1)  [left] {$b=b_{k_0}$};
\node  at (2,1)  [left] {$b_{k_1}$};
\node  at (4,1)  [left] {$b_{k_2}$};
\node  at (6,1)  [right] {$b_{k_{r-1}}$};
\node  at (8,1)  [right] {$b_{k_{r}}$};

\node at (0,2) [above left] {$bo$}; 
\node at (0,0) [below left] {$o$}; 

\foreach \x in {0,2,...,8}{
	\foreach \y in {0,2}{
\filldraw [black] (\x,\y) circle (2pt);}}

\foreach \x in {0,2}{
\draw (0,\x) -- (2,\x);
\draw (2,\x) -- (4,\x);
\draw[dashed] (4,\x) -- (6,\x);
\draw(6,\x) -- (8,\x);}

\foreach \x in {0,2,...,8}{
\draw (\x,0) --(\x,2);
}

\end{tikzpicture}
		\end{center}
	\caption{Determining the action  of $b\in B_n$ on the $A$-tree $T_m$.}\label{fig:rectangle}
	\end{figure}
    Let $L'$ be the  path in $T_m\times \{o_B\}$ corresponding to the bottom of the rectangle $R$, i.e. $L'$ is the unique edge path starting at $o$
    and whose edges have  labels  $a_{i'_1},\dots, a_{i'_r}$. The paths $bL$ and $L'$ have the same projection to $T_m$. 
    Thus $b$ fixes the projection $\pi_A(L)$ of $L$ onto $T_m$ 
    pointwise if and only if $a_{i_j}=a_{i'_j}$ for all $1\leq j\leq r$.

    By the definition of the graph $\calG_{\ualpha}$, there is a path $\ell$ in $\calG_{\ualpha}$ traversing in order the vertices $b_{k_0},\dots, b_{k_r}$. Moreover,  the edge joining $b_{k_{j-1}}$ and $b_{k_j}$  is white if and only if $a_{i_j}= a_{i'_j}$. We thus see that $b$ fixes $\pi_A(L)$ if and only if all edges of $\ell$ are white.

    By \ref{item:irr2}, \ref{item:irr2a} and \ref{item:irr2b}, there exists some $b \in B_n$ such that all edges in the closed ball $N_6(b)$ are white.
    Fix such a $b \in B_n$.
    It follows that $b$ fixes all paths of length $6$ starting at $o_A$,
    i.e.   $b\in \Gamma^{[6]}_v$. 
    Moreover,  there is some path traversing, in order, the vertices $b=b_{k_0},b_{k_1},\dots, b_{k_7}$ in $N_6(b)$ such that the edge $(b_{k_6},b_{k_7})$ is black. 
    Since, for each $j$, $(b_{k_{j-1}},b_{k_{j}})$ is an edge of $\mathcal G_{\ualpha}$,
    there is some $a_{i_j}\in A_m$ such that $\alpha_{i_j}(k_{j-1})=k_j$. 
    Consider the path $\pi_A(L)$ of $T_m$ starting at $o_A$ 
    whose  edges are sequentially labeled by $a_{i_1},\dots, a_{i_7}$. 
    Since the edge $(b_{k_6},b_{k_7})$ is black, $b$ cannot fix $\pi_A(L)$, 
    hence $b\notin \Gamma^{[7]}_v$.

    Since the $A$-tree local action is 2-transitive  and $\Gamma_v^{[6]}\not\leq \Gamma_v^{[7]}$, it follows from Theorem \ref{thm:trofweiss} that the projection of $\Gamma_{\usigma}\leq \Aut(T_m)\times \Aut(T_n)$ to $\Aut(T_m)$ is not discrete. Therefore, $\Gamma_{\usigma}$ is  irreducible by \cite[Proposition 1.2]{burgermozes2000lattices}.
\end{proof}

To prove Theorem \ref{thm:irred}, we need to show that conditions \ref{item:irr1}---\ref{item:irr3} are satisfied with high probability. 
This has already been established for  \ref{item:irr1}, \ref{item:irr2a} and \ref{item:irr3} in Lemma \ref{lem:triple_matchings}, Corollary \ref{cor: full B local action} and Theorem \ref{thm: A_side_generation}  respectively. 
Hence, all that remains is to show that conditions \ref{item:irr2} and \ref{item:irr2b} hold with high probability.
To do so, we investigate the following random variable.

\begin{definition}
	Let $m,n\in \bbN$ with $n$ even. Given  $\alpha\in \FPF_{n}$, define 
	\[P_\alpha := \{\{i,\alpha(i)\}\mid 1\leq i\leq n \}.\] 
	For random $\ualpha=(\alpha_1,\dots,\alpha_{m})\in (\FPF_{n})^m$, define the random variable  
	\[M_{n,m}(\ualpha)\coloneq  \sum_{i<j} \left \lvert P_{\alpha_i}\cap P_{\alpha_j}\right\rvert.  \] 
\end{definition}
\begin{remark}\label{rem:match_vs_blackedges}
	The number of black edges in $\calG_{\ualpha}$ is at most $M_{n,m}(\ualpha)$, with equality precisely when $\ualpha$ has no triple matchings. In particular, $\calG_{\ualpha}$ has no black edges if and only if  $M_{n,m}(\ualpha)=0$.
\end{remark}

 The following proposition demonstrates that $\calG_{\ualpha}$ has at least one black edge with sufficiently high probability:

\begin{lemma}\label{lem:positive_matchings}
	There exists a constant  $N$ such that 
	given any $m,n\in \bbN$ with  $n$ even and $n\geq N$ and
	a random $\ualpha\in (\FPF_n)^m$, 
	then the probability that $\calG_{\ualpha}$ has no black edge is at most  $\left(\frac{2}{3}\right)^{m-1}$.
\end{lemma}
\begin{proof}
    For random $\ualpha = (\alpha_1, \dots, \alpha_m) \in (\FPF_n)^m$, define the random variable 
	$Y_{n,m}(\ualpha)= \sum_{i=2}^{m} \left \lvert P_{\alpha_1}\cap P_{\alpha_{i}}\right\rvert$. 
	It follows from the definition of $M_{n,m}$ that, 
	for each $\ualpha \in (\FPF_n)^m$,
	$0\leq Y_{n,m}(\ualpha)\leq M_{n,m}(\ualpha)$. 
	Therefore $\bbP(M_{n,m}=0)\leq \bbP(Y_{n,m}=0)$.  
	
	By Corollary \ref{cor:one_third} there is a constant $N$ such that for $n\geq N$,   two random elements of $\FPF_n$ do not have a common orbit  with probability at most $2/3$, i.e.,  $\bbP(M_{n,2}=0)\leq \frac{2}{3}$. 
	Since $Y_{n,m}$ is the sum of $m-1$ non-negative independent identically distributed random variables 
	with the same distribution as $M_{n,2}$,
	we have for $n \ge N$ that
	\[\bbP(M_{n,m}=0)\leq \bbP(Y_{n,m}=0) 	\leq \left(\frac{2}{3}\right)^{m-1} \]
    The result now follows from Remark \ref{rem:match_vs_blackedges}.
\end{proof}
We now show the following:
\begin{proposition}\label{prop:white ball}
	Let $m,n\in \bbN$ with  $n> m^8$ and $n$ even. Then with probability at least  $1-\frac{1}{m}$, for some $b\in B_n$, the closed ball $N_6(b) \subset \mathcal{G}_{\ualpha}$ only contains white edges.
\end{proposition}

We first prove some lemmas that are used in the proof of Proposition~\ref{prop:white ball}.
\begin{lemma}\label{lem:pair_orbit_count}
	For any even $n$,   $\bbE(M_{n,2})=\frac{n}{2(n-1)}\leq 1$.
\end{lemma}

\begin{proof}
	Let $(\alpha,\alpha')\in (\FPF_n)^2$ be a random element.  
	As in Lemma \ref{lem:prob_match}, we may assume $\alpha$ is fixed and $\alpha'$ is chosen at random. 	
	Let $C_1, \dots, C_r$ be the orbits of $\alpha$, with $r=\frac{n}{2}$, and let $Y_i$ be the indicator random variable associated to the event that $C_i$ is an orbit of $\alpha'$. Then $M_{n,2}=\sum_{i=1}^r Y_i$, and $\bbP(Y_i=1)=\frac{1}{n-1}$ by Lemma \ref{lem:containing_transpositions}. Therefore,
	by the linearity of expectation,
	$\bbE(M_{n,2})=\frac{n}{2(n-1)}\leq 1$ as required.
\end{proof}

\begin{lemma}\label{lem:match_bound}
 For any $A>0$, 	$\bbP(M_{n,m}\geq A)\leq \frac{m^2}{2A}$.
\end{lemma}
\begin{proof}
	Observe that $M_{n,m}$ is a sum of $\frac{m(m-1)}{2}$ random	 variables, each having the same probability distribution as $M_{n,2}$. By Lemma \ref{lem:pair_orbit_count} and linearity of expectation, we see that $\bbE(M_{n,m})\leq \frac{m(m-1)}{2}\leq \frac{m^2}{2}$. The result now follows by applying Markov's inequality.
\end{proof}
\begin{proof}[Proof of Proposition \ref{prop:white ball}]
	By Lemma \ref{lem:match_bound}, $\bbP(M_{n,m}\geq \frac{m^3}{2})\leq \frac{1}{m}$. To prove Proposition \ref{prop:white ball}, it thus suffices to show that if  $M_{n,m}(\ualpha)< \frac{m^3}{2}$, then  all edges in the closed ball $N_6(b)$  are white for some vertex $b$. 
	Indeed, if $M_{n,m}(\ualpha)< \frac{m^3}{2}$ then by Remark~\ref{rem:match_vs_blackedges}, $\calG_{\ualpha}$ contains fewer than $\frac{m^3}{2}$ black edges. 
	Since vertices of  $\calG_{\ualpha}$ have valence at most $m$, 
	any edge of $\calG_{\ualpha}$
	has at most  $2m^5$ vertices a distance 5 or less from it.
	Thus there are at most
	\[2m^5\times \frac{m^3}{2}=m^8\]
	vertices that are at a distance of 5 or less from the endpoint of a black edge. 
	Hence, as $m^8 < n$, the closed ball $N_6(b)$ contains no black edge for some vertex $b$.
\end{proof}

\begin{proof}[Proof of Theorem \ref{thm:irred}]
    As noted in the paragraph after Theorem \ref{thm:irred}, we may assume $n>m^8$.
	Lemma \ref{lem:triple_matchings}, Proposition \ref{prop:white ball}, Corollary \ref{cor: full B local action}, Lemma \ref{lem:positive_matchings} and Theorem \ref{thm: A_tree_generation_alternative}  each give upper bounds for the probability than one of the conditions \ref{item:irr1}---\ref{item:irr3} in Lemma \ref{lem:irred_conditions} do not hold. Taking a union bound, we see that there is a constant $N$ such that if $n\geq \max(m^8,N)$, the probability that at least one of \ref{item:irr1}---\ref{item:irr3} is not satisfied is at most \[\frac{4m^3}{n}+\frac{1}{m}+\frac{C_1}{m}+\left(\frac{2}{3}\right)^{m-1}+\frac{C_2}{m}\] for some constants $C_1$ and $C_2$.  Therefore for some sufficiently large constant  $C$, we deduce that for all even $n > m^8$ 
	conditions  \ref{item:irr1}---\ref{item:irr3} of Lemma \ref{lem:irred_conditions} are satisfied with probability at least $1-\frac{C}{m}$. 
	Lemma \ref{lem:irred_conditions} ensures that the group $\Gamma_{\ualpha}$ satisfies 
	conclusion (\ref{it:irreducible})
	 of Theorem \ref{thm:irred} with probability at least $1-\frac{C}{m}$.
	Theorem~\ref{thm:irred} now follows from the discussion immediately after its statement.
\end{proof}

\bibliography{refs}
\bibliographystyle{amsalpha}
 
\end{document}